\documentclass{amsart}
\usepackage{amssymb}

\newtheorem{thm}{Theorem}[section]
\newtheorem{prop}[thm]{Proposition}

\newtheorem{lem}[thm]{Lemma}
\newtheorem{rema}[thm]{Remark}

\newtheorem{conj}[thm]{Conjecture}

\theoremstyle{remark}

\numberwithin{equation}{section}

\newcommand{\Z}{\mathbb{Z}}
\newcommand{\C}{\mathbb{C}}
\newcommand{\h}[0]{\mathfrak{h}} 
\newcommand{\Lnu}[0]{\hat{L}_\nu}
\begin{document}

\title[Twisted modules for free fermions and two conjectures]{On permutation-twisted free fermions and two conjectures}

\author{Katrina Barron}
\address{Department of Mathematics, University of Notre Dame,
Notre Dame, IN 46556}
\email{kbarron@nd.edu}
\author{Nathan Vander Werf}
\address{Department of Mathematics, University of Notre Dame,
Notre Dame, IN 46556}
\email{nvanderw@nd.edu}
\subjclass{Primary 17B68, 17B69, 17B81, 81R10, 81T40, 81T60}

\date{September 30, 2013}

\keywords{Vertex operator superalgebras, superconformal field 
theory}

\begin{abstract}
We conjecture that the category of permutation-twisted modules for a multi-fold tensor product vertex operator superalgebra and a cyclic permutation of even order is isomorphic to the category of parity-twisted modules for the underlying vertex operator superalgebra.  This conjecture is based on our observations of the cyclic permutation-twisted modules for free fermions as we discuss in this work, as well as previous work of the first author constructing and classifying permutation-twisted modules for tensor product vertex operator superalgebras and a permutation of odd order.   In addition, we observe that the transposition isomorphism for two free fermions corresponds to a lift of the $-1$ isometry of the integral lattice vertex operator superalgebra corresponding to two free fermions under boson-fermion correspondence.  We conjecture that all even order cyclic permutation automorphisms of free fermions can be realized as lifts of lattice isometries under boson-fermion correspondence.  We discuss the role of parity stability in the construction of these twisted modules and prove that in general, parity-unstable weak twisted modules for a vertex operator superalgebras come in pairs that form orthogonal invariant subspaces of parity-stable weak twisted modules, clarifying their role in many other settings.
\end{abstract}

\maketitle


\section{Introduction and preliminaries}

Let $V$ be a  vertex operator (super)algebra, and for a fixed positive integer $k$,  consider the tensor product vertex operator (super)algebra $V^{\otimes k}$ (see \cite{FLM3}, \cite{FHL}).  Any 
element $g$ of the symmetric group $S_k$ acts in a natural way on $V^{\otimes k}$ as a vertex operator (super)algebra automorphism, and thus it is appropriate  to consider $g$-twisted 
$V^{\otimes k}$-modules.    This is the setting for permutation orbifold conformal field theory, and for permutation orbifold superconformal field theory if the vertex operator superalgebra is not just super, but is also supersymmetric, i.e. is a representation of a Neveu-Schwarz super-extension of the Virasoro algebra.

In \cite{BDM}, the first author along with Dong and Mason constructed and classified the $g$-twisted $V^{\otimes k}$-modules for $V$ a vertex operator algebra and $g \in S_k$.  In particular, it was proved that the category of weak $(1 \; 2\; \cdots \; k)$-twisted $V^{\otimes k}$-modules is isomorphic to the category of weak $V$-modules.    In \cite{B-superpermutation-odd}, the first author extended this result to a construction and classification of $(1 \; 2 \; \cdots \; k)$-twisted $V^{\otimes k}$-modules for $V$ a vertex operator superalgebra and $k$  odd.  However, as was shown by Barron in \cite{B-superpermutation-odd}, the results of \cite{BDM} for permutation-twisted tensor product vertex operator algebras and the results of Barron for the odd order case in the super setting, do not extend in a natural way to the super setting for permutation automorphisms of even order.   Rather, the construction of the twisted modules for even order permutations is fundamentally different whenever $V$ has nontrivial odd subspace.  

In this paper, following first the construction of \cite{DZ}, we study the case of $V^{\otimes k}_{fer}$ a $k$-fold tensor product of the free fermion vertex operator superalgebra for $k$ even and $(1 \; 2 \; \cdots \; k)$ acting as a vertex operator superalgebra automorphism to gain further insight into the problem.  These models, along with observations from \cite{B-superpermutation-odd}, provide the basis for a conjecture we make on the nature of the classification of permutation-twisted modules in general.  Our main conjecture is that for $k$ even, the category of weak $(1 \; 2 \; \cdots \; k)$-twisted $V^{\otimes k}$-modules is isomorphic to the category of weak parity-twisted $V$-modules.   This contrasts to the case when $k$ is odd where, as shown in \cite{B-superpermutation-odd}, the category of $(1 \; 2 \; \cdots \; k)$-twisted $V^{\otimes k}$-modules is isomorphic to the category of weak untwisted $V$-modules.

Permutation-twisted modules for free fermions, as well as parity-twisted modules for free fermions, give examples of how ``parity-unstable" twisted modules arise.    Motivated by these examples, we prove that parity-unstable weak twisted modules arise as orthogonal pairs of invariant subspaces of a parity-stable weak twisted module.  Furthermore, if these two parity-unstable weak twisted modules are ordinary then they will always have the same graded dimension but are not isomorphic, meaning they can not be detected via techniques which only produce the graded dimensions of the twisted modules.  This simplifies much of the work in \cite{DZ2005}, \cite{DZ2010}, \cite{DH}, and shows how from a categorical standpoint, all modules should be defined so as to be parity-stable, and then what is referred to as ``parity-unstable modules" in, for instance \cite{DZ2005}, \cite{DZ2010}, \cite{DH}, should be referred to as ``parity-unstable invariant subspaces" of the parity-stable modules.

We use ``boson-fermion correspondence" to formulate another conjecture regarding whether one can realize the permutation-twisted modules for free fermions in the even cyclic case via two different constructions.  Boson-fermion correspondence refers to the fact that the two free fermion vertex operator superalgebra is isomorphic to the rank one lattice vertex operator superalgebra with length one generator, i.e. one fermion propagating on a circle.  Thus one can use the work of \cite{DZ} to construct the permutation-twisted modules for permutation-twisted free fermions, or one can try to use boson-fermion correspondence and the theory of twisted modules for a lattice vertex operator superalgebra and a lift of a lattice isometry as developed in \cite{DL2} and \cite{Xu}.  If  the automorphism on the lattice is the lift of a lattice isometry, then one has an overlap of construction techniques which can potentially give insight into the general theory of the construction and classification of twisted modules.  

However, then the question arises:  When does the permutation correspond to a lift of a lattice isometry?  In this paper we provide an example of when it does, and make a conjecture that any cyclic permutation of even order acting on free fermions is conjugate to a lift of a lattice isometry under boson-fermion correspondence.  This gives an alternative construction to even order cyclic permutation-twisted modules for free fermion vertex operator superalgebras based on twisted lattice constructions, i.e. space-time orbifold constructions, versus worldsheet orbifold constructions that has potentially for being extended in general to further explore the connection between the space-time geometry of the lattice versus the worldsheet geometry of propagating strings, following, for instance \cite{BHL}, but in a new and different setting which explicitly involves the supergeometry underlying vertex operator superalgebras.  This is particularly interesting and relevant for cases involving supersymmetric vertex operator superalgebras, cf. \cite{B-announce} -- \cite{B-mirror-maps}.  

It is important, at the same time to note the following:  We show  (see Remark \ref{doubling-k-remark} below) that the construction of twisted modules for lattice vertex operator superalgebras and a lift of a lattice isometry can not be used in general to construct the permutation-twisted modules for lattice vertex operator superalgebras directly for the case of permutations of even order.  This is because if $\nu$ is a permutation isometry on the lattice of even order $k$, due to certain properties of $\nu$, which we note in Remark \ref{doubling-k-remark},  the construction of \cite{DL2}, \cite{Xu}, only holds for $\nu$ lifted to an automorphism $\hat{\nu}$ of the vertex operator superalgebra such that $\hat{\nu}$ is of order $2k$.  Thus it is impossible for $\hat{\nu}$  to correspond to the permutation automorphism of order $k$.  Details are given in Section \ref{lattice-vosa-section} and Section \ref{general-twisted-section}, and in particular in Remark \ref{doubling-k-remark}.   In particular, we give general criteria for a lattice isometry of order $k$ to lift to a vertex operator superalgebra  automorphism of order $k$.   In the process of this, we clarify aspects of the construction of twisted modules for a lift of a lattice isometry and a positive definite integral lattice following \cite{DL2} and \cite{Xu}.

\subsection{Background}

Twisted vertex operators were discovered and used in \cite{LW}.   Twisted modules for vertex operator algebras arose in the work of I. Frenkel, J. Lepowsky and A. Meurman \cite{FLM1}, \cite{FLM2}, \cite{FLM3} for the case of a lattice vertex operator algebra and the lattice isometry $-1$, in the course of the construction of the moonshine module vertex operator algebra (see also \cite{Bo}). This structure came to be understood as an ``orbifold model" in the sense of conformal field theory and string theory.  Twisted modules are the mathematical counterpart of ``twisted sectors", which are the basic building blocks of orbifold models in conformal field theory and string theory (see \cite{DHVW1}, \cite{DHVW2}, \cite{DFMS}, \cite{DVVV}, \cite{DGM}, as well as \cite{KS}, \cite{FKS}, \cite{Ba1}, \cite{Ba2}, \cite{BHS}, \cite{dBHO}, \cite{HO}, \cite{GHHO}, \cite{Ba3} and \cite{HH}).  Orbifold theory plays an important role in conformal field theory and in superextensions, and is also a way of constructing a new vertex operator (super)algebra from a given one.  

Formal calculus arising {}from twisted vertex operators associated to a an even lattice was systematically developed in \cite{Lepowsky1985}, \cite{FLM2}, \cite{FLM3} and \cite{Le2}, and the twisted Jacobi identity was formulated and shown to hold for these operators (see also \cite{DL2}).  These results led to the introduction of the notion of $g$-twisted $V$-module \cite{FFR}, \cite{D}, for $V$ a vertex operator algebra and $g$ an automorphism of $V$.  This notion records the properties of twisted operators obtained in \cite{Lepowsky1985}, \cite{FLM1}, \cite{FLM2}, \cite{FLM3} and \cite{Le2}, and provides an axiomatic definition of the notion of twisted sectors for conformal field theory.  In general, given a vertex operator algebra $V$ and an automorphism $g$ of $V$, it is an open problem as to how to construct a $g$-twisted $V$-module.

The focus of this paper is the study of permutation-twisted sectors for free fermion vertex operator superalgebras.    A theory of twisted operators for integral lattice vertex operator superalgebras and finite automorphisms that are lifts of a lattice isometry were studied in \cite{DL2} and \cite{Xu}, and the general theory of twisted modules for vertex operator superalgebras was developed by Li in \cite{Li-twisted}.  Certain specific examples of permutation-twisted sectors in superconformal field theory have been studied from a physical point of view in, for instance, \cite{FKS}, \cite{BHS}, \cite{Maio-Schellekens1}, \cite{Maio-Schellekens2}.

The case of $V$ a vertex operator superalgebra and $V\otimes V$ being permuted by the $(1 \; 2)$ transposition is the mirror map if $V \otimes V$, in addition to being a vertex operator superalgebra is also N=2 supersymmetric (see for example, \cite{B-n2twisted}).  This is one of the motivations for studying this construction in detail in the case of free fermions.  Although the free fermion vertex operator superalgebras are not N=2 supersymmetric, they can be used to achieve supersymmetry via tensoring with an appropriate bosonic theory as, for example, in \cite{B-n2twisted}.   In particular,  a mirror-twisted module for an N=2 supersymmetric vertex operator superalgebra is naturally a representation of the ``mirror-twisted N=2 superconformal Lie superalgebra", cf. \cite{B-varna} and \cite{B-n2twisted}.   In \cite{B-mirror-maps}, further classifications and constructions involving this transposition mirror map as well as other mirror maps for N=2 supersymmetric vertex operator superalgebras arising from free fermions are given.

\subsection{The notion of vertex operator superalgebra}\label{vosa-definitions-section}

In this section, we recall the notion of vertex operator superalgebra,  following
the notation and terminology of, for instance \cite{LL}, \cite{B-vosas}. Let $x, x_0, x_1, x_2,$ etc., denote commuting independent formal variables.
Let $\delta (x) = \sum_{n \in \Z} x^n$.  We will use the binomial
expansion convention, namely, that any expression such as $(x_1 -
x_2)^n$ for $n \in \C$ is to be expanded as a formal power series in
nonnegative integral powers of the second variable, in this case
$x_2$.

A {\it vertex operator superalgebra} is a $\frac{1}{2}\Z$-graded vector space
$V=\coprod_{n\in \frac{1}{2} \Z}V_n$,
satisfying ${\rm dim} \, V < \infty$ and $V_n = 0$ for $n$ sufficiently negative, 
that is also $\Z_2$-graded by {\it sign} 
\[V = V^{(0)} \oplus V^{(1)}, \quad \mbox{with} \quad V^{(j)} = \coprod_{n \in \Z + \frac{j}{2}} V_n,\]
and equipped with a linear map
\begin{equation}
V \longrightarrow (\mbox{End}\,V)[[x,x^{-1}]], \qquad 
v \mapsto Y(v,x)= \sum_{n\in\Z}v_nx^{-n-1} ,
\end{equation}
and with two distinguished vectors ${\bf 1}\in V_0$, (the {\it vacuum vector})
and $\omega\in V_2$
(the {\it conformal element}) satisfying the following conditions for $u, v \in
V$:
\begin{eqnarray}
& &u_nv=0\ \ \ \ \ \mbox{for $n$ sufficiently large};\\
& &Y({\bf 1},x)=1;\\
& &Y(v,x){\bf 1}\in V[[x]]\ \ \ \mbox{and}\ \ \ \lim_{x\to
0}Y(v,x){\bf 1}=v;
\end{eqnarray}
\begin{multline}
x^{-1}_0\delta\left(\frac{x_1-x_2}{x_0}\right) Y(u,x_1)Y(v,x_2) \\
- (-1)^{|u||v|}
x^{-1}_0\delta\left(\frac{x_2-x_1}{-x_0}\right) Y(v,x_2)Y(u,x_1)\\
= x_2^{-1}\delta \left(\frac{x_1-x_0}{x_2}\right) Y(Y(u,x_0)v,x_2)
\end{multline}
(the {\it Jacobi identity}), where $|v| = j$ if $v \in V^{(j)}$ for $j \in \Z_2$;
\begin{equation}
[L(m),L(n)]=(m-n)L(m+n)+\frac{1}{12}(m^3-m)\delta_{m+n,0}c
\end{equation}
for $m, n\in \Z,$ where  $L(n)=\omega_{n+1}$, for $n\in \Z$,  i.e., $Y(\omega,x)=\sum_{n\in\Z}L(n)x^{-n-2}$,
and $c \in \mathbb{C}$ (the {\it central charge} of $V$);
\begin{eqnarray}
& &L(0)v=nv=(\mbox{wt}\,v)v \ \ \ \mbox{for $n \in  \frac{1}{2} \Z$ and $v\in V_n$}; \\
& &\frac{d}{dx}Y(v,x)=Y(L(-1)v,x). 
\end{eqnarray}
This completes the definition. We denote the vertex operator superalgebra just
defined by $(V,Y,{\bf 1},\omega)$, or briefly, by $V$.


Given two vertex operator superalgebras $(V_1, Y_1, \mathbf{1}^{(1)}, \omega^{(1)})$ and $(V_2, Y_2, \mathbf{1}^{(2)}, \omega^{(2)})$, we have that $(V_1 \otimes V_2, \, Y,  \, \mathbf{1}^{(1)} \otimes \mathbf{1}^{(2)}, \, \omega^{(1)} \otimes \mathbf{1}^{(2)} + \mathbf{1}^{(1)} \otimes \omega^{(2)})$ is a vertex operator superalgebra, where $Y$ is given by 
\begin{equation}\label{define-tensor-product}
Y(u_1 \otimes u_2, x) (v_1 \otimes v_2) = (-1)^{|u_2||v_1|} Y_1(u_1, x) v_1\otimes Y_2(u_2, x)v_2,
\end{equation}
for $u_1 \otimes u_2, \, v_1 \otimes v_2 \in V_1 \otimes V_2$.  

\begin{rema}\label{parity-grading-on-V}
{\em
As a consequence of the definition of vertex operator superalgebra, we have that $\mathrm{wt} (v_n u ) = \mathrm{wt} u + \mathrm{wt} v - n -1$, for $u,v \in V$ and $n \in \Z$.  This implies that $v_n \in (\mathrm{End} \, V)^{(j)}$ if and only if $v \in V^{(j)}$ for $j \in \mathbb{Z}_2$, i.e. that $v_n V^{(j)} \subseteq V^{(j + |v|) \mathrm{mod} \, 2}$.  
}
\end{rema}

\subsection{Automorphisms of vertex operator superalgebras and the notion of twisted module}\label{twisted-definitions-section}

An {\it automorphism} of a vertex operator superalgebra $V$ is a linear
automorphism $g$ of $V$ preserving ${\bf 1}$ and $\omega$ such that
the actions of $g$ and $Y(v,x)$ on $V$ are compatible in the sense
that
\begin{equation}\label{automorphism}
g Y(v,x) g^{-1}=Y(gv,x)
\end{equation}
for $v\in V.$ Then $g V_n\subset V_n$ for $n\in  \frac{1}{2} \mathbb{Z}$. 

If $g$
has finite order, $V$ is a direct sum of the eigenspaces $V^j$ of $g$,
\begin{equation}
V=\coprod_{j\in \Z /k \Z }V^j,
\end{equation}
where $k \in \mathbb{Z}_+$ is a period of $g$ (i.e., $g^k = 1$ but $k$
is not necessarily the order of $g$) and $V^j=\{v\in V \; | \; g v= \eta^j v\}$,
for $\eta$ a fixed primitive $k$th root of unity.

Note that we have the following $\delta$-function identity
\begin{equation}\label{delta-identity}
x_2^{-1} \delta \left (\frac{x_1 - x_0}{x_2} \right) \left( \frac{x_1 - x_0}{x_2} \right)^{k} = x_1^{-1} \delta \left (\frac{x_2 + x_0}{x_1} \right) \left( \frac{x_2 + x_0}{x_1} \right)^{-k}
\end{equation}
for any $k \in \mathbb{C}$.

Next we review the notions of weak, weak admissible and
ordinary $g$-twisted $V$-module for a vertex operator superalgebra $V$
and an automorphism $g$ of $V$ of finite order $k$, as well as the notion of ``parity stability" for these various kinds of $g$-twisted $V$-modules.   These are the ``standard" definitions, following, for instance \cite{DZ2005}, \cite{DZ2010}, \cite{DH}.   However, below we argue (see Remark \ref{parity-stability-remark}), that the more natural notion of ``weak $g$-twisted $V$-module" should be that of a ``weak parity-stable $g$-twisted $V$-module", and similarly for the notions of weak admissible or ordinary $g$-twisted $V$-module.

Let $(V,Y,{\bf 1},\omega)$ be a vertex operator superalgebra
and let $g$ be an automorphism of $V$ of period $k \in
\mathbb{Z}_+$. A {\it weak $g$-twisted $V$-module} is a 
vector space $M$ equipped with a linear map
\begin{equation}
V \longrightarrow  (\mbox{End}\,M)[[x^{1/k},x^{-1/k}]], \qquad
v \mapsto  Y^g (v,x)=\sum_{n\in \frac{1}{k}\Z }v_n^g x^{-n-1} ,
\end{equation}
satisfying the following conditions for $u,v\in V$ of homogeneous sign and $w\in M$:
\begin{eqnarray}
& &v_n^g w=0\ \ \ \mbox{for $n$ sufficiently large};\\
& &Y^g ({\bf 1},x)=1;
\end{eqnarray}
\begin{multline}\label{twisted-jacobi}
x^{-1}_0\delta\left(\frac{x_1-x_2}{x_0}\right)
Y^g(u,x_1)Y^g(v,x_2)\\
-  (-1)^{|u||v|} x^{-1}_0\delta\left(\frac{x_2-x_1}{-x_0}\right) Y^g
(v,x_2)Y^g (u,x_1)\\
= x_2^{-1}\frac{1}{k}\sum_{j\in \Z /k \Z}
\delta\left(\eta^j\frac{(x_1-x_0)^{1/k}}{x_2^{1/k}}\right)Y^g (Y(g^j
u,x_0)v,x_2)
\end{multline}
(the {\it twisted Jacobi identity}) where $\eta$ is a fixed primitive
$k$th root of unity.  

We denote a weak $g$-twisted $V$-module by $(M, Y^g)$, or briefly, by $M$.  

If we take $g=1$, then we obtain the notion of weak $V$-module.  Note that the notion of weak $g$-twisted $V$-module for a vertex operator superalgebra is equivalent to the notion of $g$-twisted $V$-module for $V$ as a vertex superalgebra, cf. \cite{Li-twisted}.  In particular, the term ``weak" simply implies that we are making no assumptions about a grading on $M$.

It follows from the twisted Jacobi identity that
\begin{equation}\label{cosetrelation}
Y^g(v,x)=\sum_{n\in  \Z + \frac{j}{k}}v_n^g x^{-n-1}
\end{equation}
for $j\in \Z/k\Z$ and $v\in V^j$, and thus we have 
\begin{equation}
Y^g (gv,x) = \lim_{x^{1/k} \rightarrow \eta^{-1} x^{1/k}} Y^g(v, x),
\end{equation}
where the limit stands for formal substitution. 


Let $(M_1, Y^g_1)$ and $(M_2, Y^g_2)$ be two weak $g$-twisted $V$-modules.  A {\it $g$-twisted $V$-module homomorphism} from $M_1$ to $M_2$, is a linear map $f: M_1 \longrightarrow M_2$ such that 
\begin{equation}
f(Y_1^g(v,x)w) = Y^g_2(v,x) f(w)
\end{equation}
for $v \in V$ and $w \in M_1$.

A weak $g$-twisted $V$-module may or may not have additional grading structures.  These possible grading structures fall into two different types:  1.  Those involving the $\mathbb{Z}_2$ grading structure, i.e. by sign or parity.  2.  Those involving the weight grading structure.  The first type leads to the notion of parity stability for a weak $g$-twisted $V$-module which detects whether the module has a $\mathbb{Z}_2$-grading that is compatible with the $\mathbb{Z}_2$-grading of $V$.  The second type leads to the notion of weak admissible $g$-twisted $V$-module which detects whether the module has a $\frac{1}{2k}\mathbb{Z}$-grading compatible with the $\frac{1}{2} \mathbb{Z}$-grading of $V$ where $k$ is the order of $g$.  This second type also leads to the notion of ordinary $g$-twisted $V$-modules which detects whether the $g$-twisted $V$-module is graded by eigenvectors of the twisted $L^g(0)$ operator.  

We now give the details for these different module definitions.

A {\em weak admissible} $g$-twisted $V$-module is a weak $g$-twisted
$V$-module $M$ which carries a $\frac{1}{2k}{\Z}$-grading
\begin{equation}\label{m3.12}
M=\coprod_{n\in\frac{1}{2k}\Z}M(n)
\end{equation}
such that $v^g_mM(n)\subseteq M(n+\mathrm{wt} \; v-m-1)$ for homogeneous $v\in V$, and $M(n) = 0$ for $n$ sufficiently small. 
If $g=1,$ we have the notion of weak admissible $V$-module.


An {\it ordinary $g$-twisted $V$-module}  is a weak $g$-twisted $V$-module $M$ which is $\C$-graded
\begin{equation}
M=\coprod_{\lambda \in \C}M_\lambda
\end{equation}
such that for each $\lambda$, $\dim M_{\lambda}< \infty $ and $M_{n/k
+\lambda}=0$ for all sufficiently negative integers $n$.  In addition,
\begin{equation}\label{L^g-grading}
L^g (0) w=\lambda w \qquad \mbox{for $w \in M_\lambda$},
\end{equation}
where $L^g(n) = \omega^g_{n+1}$ are the modes for the twisted vertex operator corresponding to the Virasoro element.
We will usually refer to an ordinary $g$-twisted $V$-module, as just a $g$-twisted $V$-module.  
We call a $g$-twisted $V$-module $M$ {\it simple} or {\it irreducible}
if the only submodules are 0 and $M$.  

For an ordinary $g$-twisted $V$-module, $M$, we have the notion of {\it graded dimension} or {\it $q$-dimension}, denoted $\mathrm{dim}_q M$, and defined to be
\begin{equation}
\mathrm{dim}_q M = tr_M q^{L^g(0) - c/24} = q^{-c/24} \sum_{\lambda \in \mathbb{C}} (\mathrm{dim} \, M_\lambda) q^\lambda.
\end{equation}

A weak, weak admissible or ordinary $g$-twisted $V$-module $M$ is said to be {\it parity stable} if there exists a $\mathbb{Z}_2$-grading on $M$ that is compatible with the $\mathbb{Z}_2$-grading of $V$ in the following sense:
\begin{equation}
v_m^g M^{(j)} \subseteq M^{(j + |v|) \mathrm{mod} \, 2}.
\end{equation}
In this case, setting $|w| = j$ for $ w \in M^{(j)}$, defining the parity map on $M$ by 
\begin{equation}\label{defining-sigma-M}
\sigma_M : M  \longrightarrow M, \qquad
 w \mapsto  (-1)^{|w|} w,
\end{equation}
and defining $Y^g \circ \sigma_V$ by 
\begin{equation}
Y^g \circ \sigma_V (v, x) = Y^g (\sigma_V(v), x) = (-1)^{|v|} Y^g(v, x),
\end{equation} 
we have that $(M, Y^g)$ is isomorphic to $(\sigma_M (M), Y^g \circ \sigma_V)$ as weak (or weak admissible or ordinary) $g$-twisted $V$-modules.  Note that a vertex operator superalgebra $V$ is always a parity-stable $V$-module by Remark \ref{parity-grading-on-V}.

\section{Parity-unstable modules arise as pairs of invariant subspaces of parity-stable modules}

The notion of parity stability features prominently in, for instance, \cite{DZ2005}--\cite{DZ2010}, \cite{DH}.  However, in this section, we show that all parity-unstable weak twisted modules appear as invariant subspaces of parity-stable weak twisted modules.  Thus it is enough to study the parity-stable weak twisted modules and then restrict to the invariant subspaces of such modules to study the parity-stable ones.   This theorem was motivated by constructions involving free fermions such as those given below in Sections \ref{parity-twisted-section} and \ref{permutation-dong-zhao-section}.

\begin{thm}\label{parity-stability-theorem}
Let $V$ be a vertex operator superalgebra and $g$ an automorphism.  Suppose $(M, Y_M)$ is a parity-unstable weak $g$-twisted $V$-module.  Then $(M, Y_M \circ \sigma_V)$ is a parity-unstable weak $g$-twisted $V$-module which is not isomorphic to $(M, Y_M)$.  Moreover $(M, Y_M)  \oplus (M, Y_M \circ \sigma_V)$ is a parity-stable weak $g$-twisted $V$-module.  In addition, if $(M, Y_M)$ is weak admissible or ordinary, then $(M, Y_M \circ \sigma_V)$ and hence $(M, Y_M)  \oplus (M, Y_M \circ \sigma_V)$ are weak admissible or ordinary.  In the case that  $(M, Y_M)$ is ordinary, then $(M, Y_M)$ and $(M, Y_M\circ \sigma_V)$ have the same graded dimension.  
\end{thm}

\begin{proof}
Suppose $(M, Y_M)$ is a parity-unstable weak $g$-twisted $V$-module.  Then it follows immediately that $(M, Y_M \circ \sigma_V)$ is a weak $g$-twisted $V$-module.  If $(M, Y_M \circ \sigma_V)$  were parity stable, that would imply that there exists $\sigma_M$ as in (\ref{defining-sigma-M}) such that $(\sigma_M(M), Y_M \circ \sigma_V \circ \sigma_V) = (M, Y_M)$ is isomorphic to $(M, Y_M \circ \sigma_V)$, implying $(M, Y_M)$ is parity stable.   Thus $(M, Y_M \circ \sigma_V)$ is parity unstable.

Now consider $(M, Y_M)  \oplus (M, Y_M \circ \sigma_V)$, and let 
\begin{equation}
\sigma_{M \oplus M} : M \oplus M \longrightarrow M \oplus M, \quad \sigma_{M \oplus M}: (w_1, w_2) \mapsto (w_2, w_1)
\end{equation}
so that $M \oplus M$ has a $\mathbb{Z}_2$-grading with respect to $\sigma_{M \oplus M}$ given by 
\begin{equation}
(M \oplus M)^{(0)} = \{ (w, w) \; | \; w \in M \} \quad \mbox{and} \quad (M \oplus M)^{(1)} = \{ (w, -w) \; | \; w \in M \} .
\end{equation}
Then $(\sigma_{M \oplus M} (M \oplus M), (Y_M \oplus (Y_M \circ \sigma_V)) \circ \sigma_V)$ is obviously isomorphic to $(M \oplus M, Y_M \oplus (Y_M \circ \sigma_V))$.  

The rest of the theorem follows in a straightforward way from the definitions.
\end{proof}

\begin{rema}\label{parity-stability-remark}
{\em Requiring weak twisted modules to be parity stable as part of the definition gives the more canonical notion of weak twisted module from a categorical point of view, for instance to allow for the tensor product of modules for two vertex operator superalgebras be a module for the tensor product vertex operator superalgebra.   (See e.g. (\ref{define-tensor-product})).  In particular, the notion of a weak $V$-module corresponding to a {\it representation} of $V$ as a vertex superalgebra only holds for parity-stable weak $g$-twisted $V$-modules, in that the vertex operators acting on a weak $g$-twisted $V$-module have coefficients in $\mathrm{End} \, M $ such that, the operators $v_m^g$ have a $\mathbb{Z}_2$-graded structure compatible with that of $V$.   For instance the operators $v_0^g$, for $v \in V$, give a representation of the Lie superalgebra generated by $v_0$ in $\mathrm{End} \, V$ if and only if $M$ is parity stable.  This  corresponds to $V$ acting as endomorphisms in the category of vector spaces (i.e., via even or odd endomorphisms) rather than in the category of $\mathbb{Z}_2$-graded vectors spaces (i.e., as grade-preserving and thus strictly even endomorphisms).   However, it is interesting to note that, as is shown in Section \ref{(12)-twisted-section}, for a lift of a lattice isometry, the twisted modules for a lattice vertex operator superalgebra naturally sometimes give rise to parity-unstable modules.  Thus the notion of ``parity-unstable module" does naturally arise in certain constructions.}
\end{rema}

\section{Permutation-twisted free fermion vertex operator superalgebras and a conjecture} 

We first recall the notion of free fermion vertex operator superalgebras following the notation of \cite{B-n2twisted}, but also in the spirit of \cite{DZ}.  Then we recall the construction of parity-twisted modules and construct the permutation-twisted modules following \cite{DZ}.  Finally we make a conjecture based on this example on the nature of the construction of $(1 \; 2 \; \cdots k)$-twisted $V^{\otimes k}$-modules for $k$ even and $V$ any vertex operator superalgebra based on the example of free fermions.  

\subsection{Free Fermion vertex operator superalgebras}\label{fermionic-section}  

Let $\mathfrak{h}$ be  finite-dimensional vector space over $\mathbb{C}$ equipped with a nondegenerate symmetric bilinear form $\langle \cdot , \cdot \rangle$.   Let $d$ denote the dimension of $\mathfrak{h}$, let $t$ and $x$ denote formal commuting variables, and let $U(\cdot)$ denote the universal enveloping algebra for a Lie superalgebra $(\cdot)$.  

Form the affine Lie superalgebra
\[\hat{\mathfrak{h}}^f = \mathfrak{h} \otimes t^{1/2}\mathbb{C}[t,t^{-1}]\oplus \mathbb{C} \mathbf{k} ,\]
with $\mathbb{Z}_2$-grading given by $\mathrm{sgn}(\alpha \otimes t^n)= 1$ for $ n \in  \mathbb{Z}+ \frac{1}{2}$, and $\mathrm{sgn}(\mathbf{k})= 0$, and Lie super-bracket relations 
\begin{equation}\label{commutation-fermion}
\bigl[\mathbf{k}, \hat{\mathfrak{h}}^f \bigr] = 0, \quad \mbox{and} \quad
\bigl[\alpha \otimes t^m, \beta \otimes t^n \bigr] = \langle \alpha, \beta \rangle  \delta_{m + n,0} \mathbf{k} 
\end{equation}
for $\alpha, \beta \in \mathfrak{h}$ and $m, n \in \mathbb{Z} + \frac{1}{2}$.  Then $\hat{\mathfrak{h}}^f$ is a $(( \mathbb{Z} + \frac{1}{2}) \cup \{0\})$-graded Lie superalgebra 
\[\hat{ \mathfrak{h} }^f = \coprod_{n \in (\mathbb{Z} + \frac{1}{2})\cup \{0\} }\hat{ \mathfrak{h}}^f_n \]
where $\hat{\mathfrak{h}}^f_n = \mathfrak{h} \otimes t^{-n}$, for $n \in \mathbb{Z} + \frac{1}{2}$, and $\hat{\mathfrak{h}}^f_0  = \mathbb{C}\mathbf{k}$. 
It has graded subalgebras 
\[\hat{\mathfrak{h}}^f_+ = \mathfrak{h} \otimes t^{-1/2} \mathbb{C}[t^{-1}] \quad \mathrm{and} \quad \hat{\mathfrak{h}}^f_- = \mathfrak{h} \otimes t^{1/2} \mathbb{C}[t] .\]
Note that $\hat{\mathfrak{h}}^f = \hat{\mathfrak{h}}^f_- \oplus \hat{\mathfrak{h}}^f_+ \oplus \mathbb{C} \mathbf{k}$,
and note that $\hat{\mathfrak{h}}^f$ is a Heisenberg superalgebra. 

Let $\mathbb{C}$ be the $(\hat{\mathfrak{h}}^f_- \oplus \mathbb{C}\mathbf{k})$-module such that $\hat{\mathfrak{h}}^f_-$ acts trivially and $\mathbf{k}$ acts as $1$.  Let 
\[V_{fer}^{\otimes d} = U(\hat{\mathfrak{h}}^f) \otimes_{U(\hat{\mathfrak{h}}^f_- \oplus \mathbb{C}\mathbf{k})} \mathbb{C} \simeq \mbox{$\bigwedge$}(\hat{\mathfrak{h}}^f_+)  ,\]
so that $V_{fer}^{\otimes d}$ is naturally isomorphic as a $\hat{\mathfrak{h}}^f$-module to the algebra of polynomials in the anticommuting elements of $\hat{\mathfrak{h}}^f_+$.

Let $\alpha \in \mathfrak{h}$ and $n \in \mathbb{Z} + \frac{1}{2}$.  We will use the notation
\[\alpha(n) = \alpha \otimes t^n.\]

Then $V_{fer}^{\otimes d}$ is a $\hat{\mathfrak{h}}^f$-module with action induced from the supercommutation relations (\ref{commutation-fermion}) given by
\begin{eqnarray}
\mathbf{k} \beta (-m) \mathbf{1}&=& \beta (-m) \mathbf{1} \label{fermion-action1} \\
\alpha (n) \beta (-m)\mathbf{1} &=& \langle \alpha, \beta \rangle  \delta_{m,n} \mathbf{1} \\
\alpha (-n) \beta (-m)\mathbf{1} &=& - \beta(-m) \alpha(-n)\mathbf{1} \label{fermion-action-last}
\end{eqnarray}
for $\alpha, \beta \in \mathfrak{h}$ and $m, n \in  \mathbb{N} + \frac{1}{2}$. That is letting 
$\{\alpha^{(1)}, \alpha^{(2)}, \dots, \alpha^{(d)} \}$ be an orthonormal basis for $\h$, we have
\begin{equation}
V^{\otimes d}_{fer} = \mbox{$\bigwedge$} \left[\alpha^{(j)}(-n) \; \Big| \; j = 1,\dots d, \; n \in \mbox{$\mathbb{N} + \frac{1}{2}$}\right]
\end{equation}
where $\mathbf{k}$ acts as $1$,  and for $j = 1, \dots, d$ and $n \in \mathbb{N} + \frac{1}{2}$, the operator $\alpha^{(j)} (n)$ acts as the partial derivative with respect to  $a^{(j)}(-n)$, and the operator $\alpha^{(j)} (-n)$ acts as multiplication.  

For $\alpha \in \mathfrak{h}$, set
\begin{equation} \alpha(x) = \sum_{n \in \frac{1}{2} + \mathbb{Z}} \alpha(n) x^{-n-\frac{1}{2}} ,
\end{equation}
Define the {\it normal ordering} operator ${}^\circ_\circ \cdot {}^\circ_\circ$ on products of the operators $\alpha(n)$ by
\begin{equation}
{}^\circ_\circ \alpha(m) \beta(n) {}^\circ_\circ = \left\{ \begin{array}{ll} \alpha (m) \beta (n) & \mbox{if $m\leq n$}\\
- \beta (n) \alpha (m) & \mbox{if $m> n$}\\
\end{array} \right. 
\end{equation}
for $m,n \in  \mathbb{Z} + \frac{1}{2}$. 

For $v = \alpha_1(-n_1) \alpha_2 (-n_2) \cdots \alpha_m (-n_m) \mathbf{1} \in V_{fer}$, for $\alpha_j \in \mathfrak{h}$, $n_j \in \mathbb{N}+ \frac{1}{2}$, and $j = 1, \dots, m$ and $m \in \mathbb{N}$, define the vertex operator corresponding to $v$ to be 
\begin{equation}
Y(v, x) = {}^\circ_\circ \left( \partial_{n_1-\frac{1}{2}} \alpha_1(x) \right) \left( \partial_{n_2-\frac{1}{2}} \alpha_2(x) \right) \cdots \left( \partial_{n_m-\frac{1}{2}} \alpha_m(x)\right) {}^\circ_\circ ,
\end{equation}
where for $n \in \mathbb{N}$, we use the notation $\partial_n = \frac{1}{n!} \left(\frac{d}{dx}\right)^n$.

Note that
\begin{equation}
[\alpha^{(j)} (x_1), \alpha^{(k)}(x_2)] = \delta_{j,k} \left( \frac{1}{(x_1 - x_2)} - \frac{1}{(-x_2 + x_1)} \right)
\end{equation}
implying that the $\alpha^{(j)}(x) = Y(\alpha^{(j)} (-1/2)\mathbf{1}, x)$, for $j = 1,\dots, d$, are mutually local.  And, in fact, setting 
\begin{equation}\label{fermionic-Virasoro}
 \omega = \frac{1}{2} \sum_{j = 1}^d \alpha^{(j)}(-3/2) \alpha^{(j)}(-1/2) \mathbf{1},
\end{equation}
we have that $(V^{\otimes d}_{fer}, Y, \mathbf{1}, \omega)$ is a vertex operator superalgebra with  central charge $d/2$.  $V_{fer}^{\otimes d}$ is called the {\it $d$ free fermion vertex operator superalgebra}.

When $d$ is even, $V_{fer}^{\otimes d}$ is precisely the vertex operator superalgebra studied in \cite{FFR} denoted $CM(\mathbb{Z} + \frac{1}{2})$, although in \cite{FFR} a polarized basis for $\mathfrak{h}$ is used as we will do also below, as in Equation (\ref{polarize}).

The graded dimension of $V_{fer}^{\otimes d}$ using the $\frac{1}{2}\mathbb{Z}$-grading of $V_{fer}^{\otimes d}$ by eigenvalues of $L(0)$ is
\begin{equation}
\mathrm{dim}_q V_{fer}^{\otimes d} = q^{-c/24} \sum_{n \in \frac{1}{2}\mathbb{Z}} \mathrm{dim} (V_{fer}^{\otimes d})_n q^n = q^{-d/48} \prod_{n \in \mathbb{Z}_+} (1+q^{n-1/2})^{d} = \mathfrak{f}(q)^d,
\end{equation}
where $\mathfrak{f}(q)$ is a classical Weber function \cite{YZ}.   A simple calculation shows that in fact $\mathfrak{f}(q) = \frac{\eta(q)^2}{\eta(q^2) \eta(q^{1/2})}$, where $\eta(q) = q^{1/24} \prod_{n \in \mathbb{Z}_+} (1 - q^n)$ is the Dedekind $\eta$-function.   

In addition, the {\it superdimension} of a vertex operator superalgebra $V = V^0 \oplus V^1$ is sometimes of interest.  It is defined to be $\mathrm{sdim}_q V = \mathrm{dim}_q V^{(0)} - \mathrm{dim}_q V^{(1)}$.
Thus the superdimension of $V_{fer}$ is
\begin{equation}
\mathrm{sdim}_q V_{fer} = q^{-d/48} \prod_{n \in \mathbb{Z}_+} (1-q^{n-1/2})^{d}  = \mathfrak{f}_1(q)^d
\end{equation}
where $\mathfrak{f}_1(q)$ is also a classical Weber function.  Observe that $\mathfrak{f}_1(q) = \frac{\eta(q^{1/2})}{\eta(q)}$.  

\begin{rema}\label{Weber-remark}{\em
In addition to the two classical Weber functions, $\mathfrak{f}$ and $\mathfrak{f}_1$, there is a third classical Weber function, denoted $\mathfrak{f}_2$ and given by 
\begin{equation}
\mathfrak{f}_2(q) = \sqrt{2} q^{1/24} \prod_{n \in \mathbb{Z}_+} (1+ q^n) = \sqrt{2} \frac{\eta(q^2)}{\eta(q)}.
\end{equation}
This third classical Weber function, $\mathfrak{f}_2$, will appear in Section \ref{parity-twisted-section}.  These three Weber functions, $\mathfrak{f}, \mathfrak{f}_1$, and $\mathfrak{f}_2$, form a set that is $SL_2(\mathbb{Z})$-invariant up to permutation and multiplication by $48$th roots of unity \cite{YZ}.
}
\end{rema}

Finally, we note that $V_{fer}$, and thus $V_{fer}^{\otimes d}$, is not only rational, but is self-dual as a vertex operator superalgebra (cf. \cite{FFR}, \cite{KW}, \cite{Li-untwisted}), i.e., the only irreducible $V_{fer}$-module is $V_{fer}$ itself.  
 
\subsection{Parity-twisted free fermions}\label{parity-twisted-section}

Form the affine Lie superalgebra
\[\hat{\mathfrak{h}}^f[\sigma] = \mathfrak{h} \otimes \mathbb{C}[t,t^{-1}]\oplus \mathbb{C} \mathbf{k} ,\]
with $\mathbb{Z}_2$-grading given by $\mathrm{sgn}(\alpha \otimes t^n)= 1$ for $n \in \mathbb{Z}$, and $\mathrm{sgn}(\mathbf{k})= 0$, and Lie super-bracket relations 
\begin{equation}\label{commutation-fermion-parity-twisted}
\bigl[\mathbf{k}, \hat{\mathfrak{h}}^f[\sigma] \bigr] = 0, \quad \mbox{and} \quad 
\bigl[\alpha \otimes t^m, \beta \otimes t^n \bigr] =  \langle \alpha, \beta \rangle  \delta_{m + n,0} \mathbf{k} 
\end{equation}
for $\alpha, \beta \in \mathfrak{h}$ and $m, n \in \mathbb{Z}$.  
Then $\hat{\mathfrak{h}}^f[\sigma]$ is a $\mathbb{Z}$-graded Lie superalgebra 
\[\hat{ \mathfrak{h} }^f[\sigma] = \coprod_{n \in\mathbb{Z}}\hat{ \mathfrak{h}}^f[\sigma]_n \]
where $\hat{\mathfrak{h}}^f[\sigma]_0 = \mathfrak{h} \oplus \mathbb{C} \mathbf{k}$, and  $\hat{\mathfrak{h}}^f[\sigma]_n = \mathfrak{h} \otimes t^{-n}$ for $n \neq 0$.  And $\hat{\mathfrak{h}}^f[\sigma]$ is a Heisenberg superalgebra.

If $\mathrm{dim} \, \mathfrak{h} = d$ is even, i.e. $d = 2l$, then we can choose a polarization of $\mathfrak{h}$ into maximal isotropic subspaces $\mathfrak{a}^\pm$.  That is $\mathfrak{a}^\pm$ both have dimension $l$, and satisfy $\langle \mathfrak{a}^+, \mathfrak{a}^+ \rangle = \langle \mathfrak{a}^-, \mathfrak{a}^-\rangle = 0$, and we can choose a basis of $\mathfrak{a}^-$, given by $\{\beta^{(1)}_-, \beta^{(2)}_-, \dots, \beta^{(l)}_-\}$, and a dual basis for $\mathfrak{a}^+$, given by $\{\beta^{(1)}_+, \beta^{(2)}_+, \dots, \beta^{(l)}_+\}$ such that $\langle \beta^{(j)}_-, \beta^{(n)}_+ \rangle = \delta_{j,n}$.

 If $\mathrm{dim} \, \mathfrak{h} = d$ is odd, i.e. $d = 2l+1$, then we can choose a polarization of $\mathfrak{h}$ into maximal isotropic subspaces $\mathfrak{a}^\pm$, each of  dimension $l$, and a one-dimensional space $\mathfrak{e}$, so that $\mathfrak{h} = \mathfrak{a}^- \oplus \mathfrak{a}^+ \oplus \mathfrak{e}$, and such that $\langle \mathfrak{a}^\pm, \mathfrak{e}\rangle = 0$, and $\mathfrak{e} = \mathbb{C} \epsilon$ with $\langle \epsilon , \epsilon \rangle = 2$. 

\begin{rema}\label{polarization-remark} \em{If $\{ \alpha^{(1)}, \alpha^{(2)}, \dots, \alpha^{(d)} \}$ is an orthonormal basis for $\mathfrak{h}$ with respect to the symmetric bilinear form, then a polarization for $\mathfrak{h}$ can be given as follows: For
$d$ either $2l$ or $2l +1$, set 
\begin{equation}\label{polarize}
\beta^{(j)}_\pm = \frac{1}{\sqrt{2}} \left( \alpha^{(j)} \pm i \alpha^{(j + l)} \right)
\end{equation} 
for $j =1,2,\dots l$.  Then $\mathfrak{a}^\pm = \mathrm{span}_\mathbb{C} \{  \beta^{(1)}_\pm, \beta^{(2)}_\pm, \dots, \beta_\pm^{(l)} \}$ gives a decomposition into maximal polarized spaces.  If $d = 2l+1$, then set $\epsilon = \sqrt{2} \alpha^{(d)}$.   Note that (\ref{polarize}) is equivalent to $\alpha^{(j)} = \frac{1}{\sqrt{2}} \left( \beta_+^{(j)} + \beta_-^{(j)}\right)$ and  $\alpha^{(j+l)} = \frac{-i}{\sqrt{2}} \left( \beta_+^{(j)} - \beta_-^{(j)}\right)$ for $j = 1,\dots,l$.
}
\end{rema}

Then $\hat{\mathfrak{h}}^f[\sigma]$ has the following graded subalgebras
\[\hat{\mathfrak{h}}^f [\sigma]_+ = \mathfrak{h} \otimes t^{-1} \mathbb{C}[t^{-1}] \qquad \mbox{and} \qquad \hat{\mathfrak{h}}^f [\sigma]_- = \mathfrak{h} \otimes t \mathbb{C}[t], \]
and we have $\hat{\mathfrak{h}}^f[\sigma] = \hat{\mathfrak{h}}^f[\sigma]_- \oplus \mathfrak{h} \oplus \hat{\mathfrak{h}}^f[\sigma]_+ \oplus \mathbb{C}\mathbf{k}$.  In addition, $\hat{\mathfrak{h}}^f[\sigma]$ has the subalgebras
\[\hat{\mathfrak{h}}^f[\sigma]_+ \oplus \mathfrak{a}^+ \qquad \mbox{and} \qquad  \hat{\mathfrak{h}}^f[\sigma]_- \oplus \mathfrak{a}^-\]
for $d$ even and
\[\hat{\mathfrak{h}}^f[\sigma]_+ \oplus \mathfrak{a}^+ \oplus \mathfrak{e} \qquad \mbox{and} \qquad  \hat{\mathfrak{h}}^f[\sigma]_- \oplus \mathfrak{a}^-\]
for $d$ odd.

Let $\mathbb{C}$ be the $(\hat{\mathfrak{h}}^f [\sigma]_- \oplus \mathfrak{a}^- \oplus  \mathbb{C} \mathbf{k})$-module such that $\hat{\mathfrak{h}}^f[\sigma]_-\oplus \mathfrak{a}^-$ acts trivially and $\mathbf{k}$ acts as $1$.  Set
\begin{equation}
M_\sigma =  U(\hat{\mathfrak{h}}^f[\sigma]) \otimes_{U(\hat{\mathfrak{h}}^f [\sigma]_- \oplus  \mathfrak{a}^- \oplus \mathbb{C} \mathbf{k})} \mathbb{C}.
\end{equation}
Then as a vector space, we have
\begin{equation}
M_\sigma \stackrel{\mathrm{vec.sp.}}{\simeq}\left\{ \begin{array}{ll}
\mbox{$\bigwedge$}(\hat{\mathfrak{h}}^f [\sigma]_+ \oplus \mathfrak{a}^+) & \mbox{if $d$ is even}\\
\mbox{$\bigwedge$}(\hat{\mathfrak{h}}^f [\sigma]_+ \oplus \mathfrak{a}^+ \oplus  \mathfrak{e}) & \mbox{if $d$ is odd}
\end{array} \right.,
\end{equation}
where if $d$ is even, this is also an associative algebra isomorphism, but if $d$ is odd it is not;  rather, if $d$ is odd, $M_\sigma$ is a Clifford algebra but not an exterior algebra.

Let $\alpha \in \mathfrak{h}$ and $n \in \mathbb{Z}$.  We use the notation 
\[ \overline{\alpha(n)} = \alpha \otimes t^n \in \hat{\mathfrak{h}}^f[\sigma] \]
where the overline is meant to distinguish elements of $\hat{\mathfrak{h}}^f[\sigma]$ from elements of $\hat{\mathfrak{h}}$, used to construct the free bosonic theory.

Then $M_\sigma$ is a $\hat{\mathfrak{h}}^f[\sigma]$-module.  For $d$ even, the action induced from the supercommutation relations (\ref{commutation-fermion-parity-twisted}) is given by 
\begin{eqnarray}
\mathbf{k} \overline{\beta (-m)} \mathbf{1} &=& \overline{\beta (-m)} \mathbf{1} \label{twisted-relations1}\\
\overline{\alpha  (n)} \, \overline{\beta (-m) } \mathbf{1} &=& \langle \alpha, \beta \rangle  \delta_{m,n} \mathbf{1} \\
\overline{\alpha (-n)} \, \overline{\beta (-m)} \mathbf{1} &=& - \overline{\beta(-m)} \, \overline{\alpha(-n)}  \mathbf{1}
\end{eqnarray}
for either (i) $\alpha, \beta \in \mathfrak{h}$ and $m, n \in \mathbb{Z}_+$; (ii) $\alpha \in \mathfrak{h}$, $\beta \in \mathfrak{a}^+$, $m =0$, and $n \in \mathbb{Z}_+$;  or (iii) $\alpha \in \mathfrak{a}^-$, $\beta \in \mathfrak{h}$, $n = 0$, and $m \in \mathbb{Z}_+$; and
\begin{equation}\label{twisted-relations-last}
\overline{\alpha(0)} \, \overline{\beta(0)} \mathbf{1} \ = \ \langle \alpha, \beta \rangle \mathbf{1}
\end{equation}
if $\alpha \in \mathfrak{a}^-$ and $\beta \in \mathfrak{a}^+$, and where here $\mathbf{1} = \mathbf{1}_{M_\sigma} = 1$.  

For $d$ odd, the action induced from the supercommutation relations are given by (\ref{twisted-relations1})--(\ref{twisted-relations-last}) as well as  
\begin{eqnarray}
\mathbf{k} \overline{\epsilon(0)} \mathbf{1} = \overline{\epsilon(0)}\mathbf{1}, \quad & &  \quad  
\overline{\alpha(0)} \, \overline{\epsilon(0)} \mathbf{1} = 0,\\
\overline{\beta(0)} \,  \overline{\epsilon(0)} \mathbf{1}  = -  \overline{\epsilon(0)}\,  \overline{\beta(0)}\mathbf{1}, & & 
\overline{\epsilon(0)} \, \overline{\epsilon (0)} = \frac{1}{2} \langle \epsilon, \epsilon \rangle \mathbf{1},
\end{eqnarray}
for $\alpha \in \mathfrak{a}^-$, $\beta \in \mathfrak{a}^+$, and $\epsilon \in \mathfrak{e}$. 

In particular, letting $\{\beta^{(1)}_\pm, \beta^{(2)}_\pm, \dots, \beta^{(l)}_\pm\}$ be the bases for the polarized spaces $\mathfrak{a}^\pm$ as defined in Remark \ref{polarization-remark}, and if $d$ is odd, letting $\mathfrak{e} = \mathbb{C} \epsilon$ with $\langle \epsilon,\epsilon \rangle = 2$, then we have 
\begin{equation*}
M_\sigma = \mbox{$\bigwedge$} \Big[\overline{\beta^{(j)}_-(-m)}\mathbf{1}, \, \overline{\beta^{(j)}_+(-n)} \mathbf{1} \, \Big| \, \mbox{ $m \in \mathbb{Z}_+$, $n \in \mathbb{N}$, and $j = 1, \dots, l$} \Big] 
\end{equation*}
for $d$ even, and in this case, the identification is as an associative algebra.  For $d$ odd, we have
\begin{equation*}
M_\sigma = \mbox{$\bigwedge$} \Big[\overline{\beta^{(j)}_-(-m)}\mathbf{1}, \, \overline{\beta^{(j)}_+(-n)}\mathbf{1}, \, \overline{\epsilon(-n)}\mathbf{1} \, \Big| \, \mbox{$m \in \mathbb{Z}_+$, $n \in \mathbb{N}$, and $j = 1, \dots, l$} \Big] 
\end{equation*}
where in this case, the identification is as a vector space but not as an associative algebra.  As an associative algebra with identity, $M_\sigma$ for $d$ odd is the Clifford algebra generated by $\hat{\mathfrak{h}}^f [\sigma]_+ \oplus \mathfrak{a}^+ \oplus  \mathfrak{e}$ with the corresponding symmetric bilinear form.  

That is, for $d$ even, $k$ acts as $1$, and for $j = 1, \dots, l$, and $n \in \mathbb{Z}_+$, the operator $\overline{\beta^{(j)}_\pm (n) }$  acts as the partial derivative with respect to $\overline{\beta^{(j)}_\mp (-n) }$, the operator $\overline{\beta^{(j)}_\pm (-n) }$ acts as multiplication by $\overline{\beta^{(j)}_\pm (-n) }$, the operator $\overline{\beta^{(j)}_- (0) }$ acts as the partial derivative with respect to $\overline{\beta^{(j)}_+ (0) }$, and the operator
$\overline{\beta^{(j)}_+ (0) }$ acts via multiplication.  If $d$ is odd, then we have the operators as above in addition to the operators $\overline{\epsilon (n)}$ for $n \in \Z_+$, which act as two times the partial derivative with respect to $\overline{\epsilon (-n)}$, and the operators $
\overline{\epsilon (-n)}$ for $n \in \mathbb{N}$, which act as multiplication  with the condition that $\overline{\epsilon(0)} \overline{\epsilon (0)} = 1$.

For $\alpha \in \mathfrak{h}$, set
\begin{equation} \alpha(x)^\sigma = \sum_{n \in \mathbb{Z}} \overline{\alpha (n) } x^{-n-\frac{1}{2}} .
\end{equation}
Then for the orthonormal basis of $\mathfrak{h}$, $\alpha^{(j)}$, for $j = 1, \dots, d$, we have 
\begin{equation}
[\alpha^{(j)} (x_1)^\sigma, \alpha^{(k)}(x_2)^\sigma ] = \delta_{j,k} \, x_1^{1/2}x_2^{-1/2}\left( \frac{1}{(x_1 - x_2)} - \frac{1}{(-x_2 + x_1)} \right)\\
\end{equation}
for $j,k = 1,\dots,d$, implying that the $\alpha^{(j)}(x)^\sigma$, for $j = 1, \dots, d$, are mutually local. 

For $v \in V_{fer}^{\otimes d}$, define $Y^\sigma(v,x) : M_\sigma \longrightarrow M_\sigma [[x^{1/2}, x^{-1/2}]]$ as follows: For $\alpha \in \mathfrak{h}$, $n \in \mathbb{N} + 1/2$, and $u \in V_{fer}^{\otimes d}$, let
\begin{multline}\label{define-sigma-twisted}
Y^\sigma (\alpha(-n) u,x) = Y^\sigma(\alpha_{-n -1/2} u,x) = \mathrm{Res}_{x_1} \mathrm{Res}_{x_0} \left( \frac{x_1 - x_0}{x} \right)^{1/2} x_0^{-n-1/2}\\
\cdot
\left(  x^{-1}_0\delta\left(\frac{x_1-x}{x_0}\right)
\alpha (x_1)^\sigma Y^\sigma(u,x)
-  (-1)^{|u|} x^{-1}_0\delta\left(\frac{x-x_1}{-x_0}\right) Y^\sigma
(u,x) \alpha (x_1)^\sigma  \right) .
\end{multline}
Then since $V_{fer}^{\otimes d} = \langle \alpha^{(j)}(-1/2) \mathbf{1} \; | \; j = 1,\dots, d \rangle$, equation (\ref{define-sigma-twisted}) defines $Y^\sigma(v,x)$ iteratively for any $v \in V_{fer}^{\otimes d}$. 

Recalling that the Virasoro element, $\omega_{fer}$, for the free fermionic vertex operator superalgebra $V_{fer}^{\otimes d}$ is given by (\ref{fermionic-Virasoro}), we have 
\begin{equation}
Y^\sigma(\omega_{fer}, x) = \frac{1}{2} \sum_{j = 1}^d Y^\sigma( \alpha^{(j)} (-1/2)_{-2} \alpha^{(j)}(-1/2)\mathbf{1},x)  = \sum_{n \in \mathbb{Z}}  L^\sigma(n) x^{-n-2}, 
\end{equation}
and thus
\begin{equation}\label{L-sigma-twisted}
L^\sigma(m) = \sum_{j = 1}^d \sum_{{n \in \mathbb{Z}}\atop{n > -\frac{m}{2}}} \left(n + \frac{m}{2}\right) \overline{\alpha^{(j)} (-n)} \, \overline{\alpha^{(j)} (n +m) }  + \frac{d}{16} \delta_{m,0}.
\end{equation}
From this it follows that $\left[L^\sigma(-1), Y^\sigma( \alpha^{(j)} (-1/2)\mathbf{1},x )\right] = \frac{d}{dx} Y^\sigma (\alpha^{(j)}(-1/2)\mathbf{1}, x)$.
Thus from \cite{Li-twisted}, we have that $M_\sigma$ is a weak $\sigma$-twisted module for $V_{fer}^{\otimes d}$.  It is also admissible.  In \cite{FFR}, if $d$ is even, $M_\sigma$ is denoted by $CM(\mathbb{Z})$.  

By \cite{Li-twisted} as well as \cite{DZ}, in the case that $d = \mathrm{dim} \, \mathfrak{h}$ is even, $M_\sigma$ is irreducible and is the only irreducible admissible $\sigma$-twisted module for $V_{fer}^{\otimes d}$, up to isomorphism.   It is parity stable and is also an ordinary $\sigma$-twisted $V_{fer}^{\otimes d}$-module, as we will see below when we discuss the $L^\sigma(0)$-grading and the $\mathbb{Z}_2$-grading.

In the case that $d$ is odd, $M_\sigma$ is irreducible as a parity-stable module but reduces to the direct sum of two irreducible parity-unstable subspaces, and these two are the only irreducible admissible parity-unstable $\sigma$-twisted modules for $V_{fer}^{\otimes d}$, up to isomorphism.   In this case, setting
\[ W = \mbox{$\bigwedge$} \left[\overline{\beta^{(j)}_-(-m)} \mathbf{1}, \, \overline{\beta^{(j)}_+(-n)}\mathbf{1}, \,  \overline{\epsilon (-m) }\mathbf{1} \, \Big| \, \mbox{$m \in \mathbb{Z}_+$, $n \in \mathbb{N}$, and $j = 1, \dots, l$} \right] \]
and letting $W = W^0 \oplus W^1$ be the decomposition of $W$ into even and odd subspaces, 
these two irreducibles are given by
\begin{eqnarray}
M_\sigma^\pm = \left(1 \pm  \, \overline{\epsilon(0)}1\right) W^0 \oplus \left(1 \mp \, \overline{\epsilon(0)}1 \right) W^1,
\end{eqnarray}
and we have $M_\sigma = M_\sigma^- \oplus M_\sigma^+$.  That $M_\sigma^\pm$ are in fact ordinary $\sigma$-twisted modules for $V_{fer}^{\otimes d}$ and parity-unstable,  we shall see now by discussing the $L^\sigma(0)$-grading and the $\mathbb{Z}_2$-grading.

In terms of the polarization of $\mathfrak{h}$ with respect to the basis $\alpha^{(j)}$, we have from equation (\ref{L-sigma-twisted})
\begin{equation} 
L^{\sigma}(0)  =  \sum_{j = 1}^l  \sum_{m  \in \mathbb{Z}_+} \left( m  \overline{\beta^{(j)}_+( -m) } \, \overline{\beta_-^{(j)} (m) } + m  \overline{\beta^{(j)}_-( -m) } \, \overline{\beta_+^{(j)} (m) } \right)    + L' + \frac{d}{16}
\end{equation}
where if $d$ is even, $L' = 0$, and if $d$ is odd, $L' = \frac{1}{2}\sum_{m\in \mathbb{Z}_+} m \overline{\epsilon (-m)} \, \overline{\epsilon(m) }$.
Thus for $j = 1,\dots, l$, and $m \in \mathbb{Z}_+$, the $L^\sigma(0)$ grading is given by 
\begin{equation}
\mathrm{wt} \, \mathbf{1} =  \mathrm{wt} \, \overline{\beta^{(j)}_+ (0) }\mathbf{1} = \frac{d}{16},  \quad \mathrm{and} \quad  \mathrm{wt} \, \overline{\beta^{(j)}_\pm (-m)} \mathbf{1} = m + \frac{d}{16}, 
\end{equation}
for $d = 2l$, and if $d$ is odd, we also have 
\begin{equation}
 \mathrm{wt} \,  \overline{\epsilon(0) } \mathbf{1} =  \frac{d}{16}, \quad  \mathrm{and} \quad  \mathrm{wt} \, \overline{\epsilon (-m) } \mathbf{1} =  m +\frac{d}{16}.
\end{equation}

Therefore, for $d$ even, the graded dimension of $M_\sigma$ is
\begin{equation}
\mathrm{dim}_q M_\sigma =  q^{-c/24} \sum_{\lambda \in \mathbb{C}} (M_\sigma)_\lambda q^\lambda  \ = \ q^{-d/48} q^{d/16} 2^{d/2} \prod_{n\in \mathbb{Z}_+} (1 + q^n)^d =  \mathfrak{f}_2 (q)^d 
\end{equation} 
where $\mathfrak{f}_2$ is a classical Weber function as discussed in Remark \ref{Weber-remark}.  
For $d$ odd, the graded dimension of $M_\sigma$ is 
\begin{equation}
\mathrm{dim}_q M_\sigma \ =\  q^{-d/48} q^{d/16} 2^{(d+1)/2} \prod_{n\in \mathbb{Z}_+} (1 + q^n)^d \ = \ \sqrt{2} \mathfrak{f}_2 (q)^d ,
\end{equation} 
and the grading of each of the two submodules $M^\pm_\sigma$ is exactly half that of the graded dimension of $M_\sigma$.  

\begin{lem}  If $d$ is even, then the unique up to equivalence irreducible parity-twisted module for $d$ free fermions $M_\sigma$ is a parity-stable  twisted module.   If $d$ is odd, then the two unique up to equivalence  irreducible parity-twisted modules $M_\sigma^\pm$ for $d$ free fermions are parity-unstable invariant subspace of $M_\sigma$.  In addition, $M_\sigma = M_\sigma^+ \oplus M_\sigma^-$ is a parity-stable twisted module and is irreducible as a parity-stable twisted module. 
\end{lem}

\begin{proof} We first show that $M_\sigma$ is parity stable for $d$ even or odd.   Define a $\mathbb{Z}_2$-grading on $M_\sigma$ via the natural $\mathbb{Z}_2$-grading on $\bigwedge(\hat{\mathfrak{h}}[\sigma]_+ \oplus \mathfrak{a}^+)$ for $d$ even and via the natural $\mathbb{Z}_2$-grading on $\bigwedge(\hat{\mathfrak{h}}[\sigma]_+ \oplus \mathfrak{a}^+ \oplus \mathfrak{e})$ for $d$ odd.   That is $w = \beta_1(-n_1) \beta_2(-n_2) \cdot \beta_m (-n_m)1 \in M^\sigma$ for $n_j \in \mathbb{Z}_+$ if $\beta_j \in \hat{\mathfrak{h}}[\sigma]_+$ and $n_j \in \mathbb{N}$ if $\beta_j \in \mathfrak{a}^+ \oplus \mathfrak{e}$ has odd parity if $m$ is odd and even parity if $m$ is even.   

Then $Y^\sigma (\alpha^{(j)} (-1/2)\mathbf{1}, x) \cdot  w = \alpha^{(j)} (x)^\sigma \cdot w = \sum_{n \in \mathbb{Z}} \overline{\alpha(n)} \cdot w x^{-n-1/2}$ is contained in $M_\sigma^{(m + 1) \mathrm{mod} \, 2} [[x^{1/2}, x^{-1/2}]]$, implying $v_n^\sigma \cdot M^{(j)}_\sigma \subset  M^{(j + |v|) \mathrm{mod} \, 2}_\sigma$ for all $v \in V_{fer}^{\otimes d}$ and $w \in M_\sigma$.

However, for $d$ odd, considering the irreducible modules $M^\pm_\sigma$, we have, for instance 
\begin{multline}
Y^\sigma( \epsilon(-1/2) \mathbf{1}, x) \cdot (1 \pm \overline{\epsilon (0)}1) \\
= \pm (1 \pm \overline{\epsilon(0)}1) x^{-1/2} + (1 \mp \overline{\epsilon(0)}1) \sum_{-n\in \mathbb{Z}_+} \overline{\epsilon(n)}1 x^{ -n -1/2}.
\end{multline}
Thus there exists no $\mathbb{Z}_2$-grading on $M_\sigma^\pm$ such that this lowest weight vector $(1 \pm \overline{\epsilon (0)}1)$ has a parity compatible with the $\mathbb{Z}_2$-grading of $V_{fer}^{\otimes d}$.  
\end{proof}

From Theorem \ref{parity-stability-theorem}, we have that $(M^-_\sigma, Y^\sigma \circ \sigma)$ must be a parity-unstable parity-twisted module that is isomorphic to $(M^+_\sigma, Y^\sigma)$.  This isomorphism is given explicitly by 
\begin{eqnarray}
f : M^+_\sigma &\longrightarrow& M^-_\sigma \\
 \left((1 +  \overline{\epsilon(0)}1) w_0, (1 - \overline{\epsilon(0)}1) w_1\right) &\mapsto&  \left((1 -  \overline{\epsilon(0)}1) w_0, - (1 + \overline{\epsilon(0)}1) w_1\right), \nonumber
\end{eqnarray}
for $w_j \in W^{(j)}$ for $j = 0,1$.

\subsection{Permutation-twisted modules for free fermions}\label{permutation-dong-zhao-section}

Now we turn our attention to tensor product vertex operator superalgebras.
Let $V=(V,Y,{\bf 1},\omega)$ be a vertex operator superalgebra, and let $k$ be a
fixed positive integer.  Then $V^{\otimes k}$ is also a vertex operator superalgebra, and the permutation group $S_k$ acts naturally on $V^{\otimes k}$ as signed automorphisms.  That is
$(j \; j+1) \cdot (v_1 \otimes v_2 \otimes \cdots \otimes v_k) = (-1)^{|v_j||v_{j+1}|} (v_1 \otimes v_2 \otimes \cdots v_{j-1} \otimes v_{j+1} \otimes v_j \otimes v_{j+2} \otimes \cdots \otimes v_k)$, and we take this to be a left action so that, for instance
\begin{eqnarray}
\qquad \ \ (1 \; 2 \cdots  k) : V \otimes V \otimes \cdots \otimes V \! \! \! \! &\longrightarrow & \! \! \! \! V \otimes V \otimes \cdots \otimes V\\
v_1 \otimes v_2 \otimes \cdots \otimes v_k \! \! \! \! \! \! & \mapsto & \! \! \! \! \!  \! (-1)^{|v_1|(|v_2| + \cdots + |v_k|)} v_2 \otimes v_3 \otimes \cdots \otimes v_k \otimes v_1. \nonumber
\end{eqnarray}

Letting $V = V_{fer}$, we have that $g=(1 \; 2 \cdots  k)$ acting as a signed permutation on $V_{fer}^{\otimes k}$ is a lift of the following permutation on $\mathfrak{h}$, the $k$-dimensional Heisenberg Lie superalgebra used to construct $V_{fer}^{\otimes k}$:  Let $\alpha^{(j)}$, for $j = 1,\dots, d=k$ be an orthonormal basis for $\h$ as before.  Then 
\begin{eqnarray}
\qquad \ \ (1 \; 2 \cdots  k) : \mathbb{C}\alpha^{(1)} \oplus \mathbb{C}\alpha^{(2)} \oplus \cdots \oplus \mathbb{C}\alpha^{(k)} \! \! \! \! &\longrightarrow & \! \! \! \! \mathbb{C}\alpha^{(1)} \oplus \mathbb{C}\alpha^{(2)} \oplus \cdots \oplus \mathbb{C}\alpha^{(k)}\\
(c_1, c_2, \dots, c_k) & \mapsto &  (c_2, c_3, \dots, c_k, c_1), \nonumber
\end{eqnarray}
that is $(1 \; 2 \cdots  k) \alpha^{(j)} = \alpha^{(j-1)}$ for $j = 1, \dots, k$ where $\alpha^{(-1)}$ is understood to be $\alpha^{(k)}$.

Defining $\mathfrak{h}^{0} = \{ h \in \mathfrak{h} \; | \; gh = h \}$, we have that $\mathfrak{h}^{0} = \mathbb{C}\beta$ with $\beta = \sum_{j =1}^k \alpha^{(j)}$.  

Thinking of $\mathfrak{h}$ as a purely odd super vector space, we also have a parity map on $\mathfrak{h}$ denoted by $\sigma_{\mathfrak{h}}$ which of course just acts as multiplication by $-1$.  Then defining $\mathfrak{h}^{0*} = \{ h \in \mathfrak{h} \; | \; g\sigma_{\mathfrak{h}} h = h \}$, we have that $\mathfrak{h}^{0*} = \{ (c_1, c_2, \dots, c_k) \in \mathfrak{h} \; | \; c_j = - c_{j+1} \; \mbox{for $1 \leq j \leq k-1$ and} \; c_k = -c_1 \}$.  So that $\mathfrak{h}^{0*}$ is of dimension 1 if $k$ is even and is of dimension 0 if $k$ is odd.

Since from the first author's work in \cite{B-superpermutation-odd}, we already have a unified construction and classification of all $(1 \; 2 \; \cdots \; k)$-twisted $V^{\otimes k}$ modules for $k$ odd and $V$ any vertex operator superalgebra, we turn our attention here to the case when $k$ is even, following \cite{DZ}.  In this case, according to \cite{DZ}, we should obtain two equivalence classes of irreducible parity-unstable $(1 \; 2 \; \cdots \; k)$-twisted $V^{\otimes k}_{fer}$ modules for $k$ even.  

Letting $g = (1 \; 2 \; \cdots \; k)$ for $k$ odd, we consider the $g\sigma_{\mathfrak{h}}$-eigenspaces
\begin{equation}\label{hnfgrading}
\h_{(n)}^f = \{ h \in \h \; | \; g\sigma_\h h = \eta^{n} h \} \subset \h,
\end{equation}
for $\eta$ a fixed primitive $k$th root of unity.  And so in terms of our discussion above, $\h^{0*} = \h_{(0)}^f$.   (In the notation of \cite{DZ}, we have $\h^f_{(n)} = H^{n*}$.)  We use the $f$ superscript to denote this fermionic setting as opposed to the bosonic setting of the lattice we will encounter latter in Section \ref{general-twisted-section}.

We have $\h = \coprod_{n \in \mathbb{Z}/k\mathbb{Z}} \h^f_{(n)}$, where
we identify $\h^f_{(n \; \mathrm{mod} \; k)}$ with $\h^f_{(n)}$ for $n \in
\mathbb{Z}$.
For $n \in \mathbb{Z}/k\mathbb{Z}$, denote by $P_n : \h \longrightarrow \h^f_{(n)}$,
the projection onto $\h^f_{(n)}$, and for $h \in \h$ and $n \in
\mathbb{Z}$, set $h_{(n)} = P_{(n \; \mathrm{mod}\; k)} h$.  In
general, we have that for $h \in \h$ and $n \in \Z$,
\begin{equation}\label{h_n-formula-fermion}
h_{(n)} = \frac{1}{k} \sum_{j=0}^{k-1} \eta^{-nj} (g \sigma_\h)^j h.
\end{equation}
Then it is clear that $\mbox{dim }{\h^f_{(n)}}=1$, for $0\leq n\leq k-1$.   In fact, $\alpha_{(n)}^{(1)}$ can be taken as a basis for each $\h^f_{(n)}$.

 Viewing $\h$ as an abelian Lie superalgebra concentrated in the odd component, let
\begin{equation}
\hat{\h}^f[g] = \coprod_{n\in\frac{1}{k} \Z} \h^f_{(kn)}\otimes
t^{n}\oplus \C {\bf k}
\end{equation}
with $\mathbb{Z}_2$-grading given by $\mathrm{sgn}(\alpha \otimes t^n)= 1$ for $n \in \frac{1}{k}\mathbb{Z}$, and $\mathrm{sgn}(\mathbf{k})= 0$, and Lie super-bracket relations 
\begin{equation} \label{commutation-dong-zhao}
\bigl[\mathbf{k}, \hat{\mathfrak{h}}^f[g] \bigr] =  0,  \quad \mbox{and} \quad
\bigl[\alpha \otimes t^m, \beta \otimes t^n \bigr] =  \langle \alpha, \beta \rangle  \delta_{m + n,0} \mathbf{k}
\end{equation}
for $\alpha \in{\h^f}_{(km)}$, $\beta \in{\h^f}_{(kn)}$, and
$m,n\in\frac{1}{k} \Z$.
Then $\hat{\mathfrak{h}}^f[g]$ is a $\frac{1}{k}\mathbb{Z}$-graded Lie superalgebra 
\[\hat{ \mathfrak{h} }^f[g] = \coprod_{n \in\frac{1}{k}\mathbb{Z}}\hat{ \mathfrak{h}}^f[g]_n \]
where $\hat{\mathfrak{h}}^f[g]_0 = \mathfrak{h}^f_{(0)} \oplus \mathbb{C} \mathbf{k}$, and  $\hat{\mathfrak{h}}^f[g]_n = \mathfrak{h}_{(kn)} \otimes t^{-n}$ for $n \neq 0$.  And $\hat{\mathfrak{h}}^f[g]$ is a Heisenberg superalgebra. 

Then $\hat{\h}^f[g]$ has the following graded subalgebras
\begin{equation}
\hat{\h}^f[g]_{+}=\coprod_{n<0}{\h^f}_{(kn)}\otimes t^{n}, \quad \mbox{and} \quad
\hat{\h}^f[g]_{-}=\coprod_{n>0}{\h^f}_{(kn)}\otimes t^{n},
\end{equation}
and we have $\hat{\h}^f[g] = \hat{\h}^f[g]_{-} \oplus \h_{(0)}^f \oplus \hat{\h}^f[g]_+.$

Let $\mathbb{C}$ be the $\hat{\mathfrak{h}}^f [g]_-\oplus \mathbb{C} \mathbf{k}$-module such that $\hat{\mathfrak{h}}^f[g]_-$ acts trivially and $\mathbf{k}$ acts as $1$.  Set
\begin{equation}
M_g =  U(\hat{\mathfrak{h}}^f[g]) \otimes_{U(\hat{\mathfrak{h}}^f [g]_- \oplus  \mathbb{C} \mathbf{k})} \mathbb{C}.
\end{equation}
Then as a vector space, we have
\begin{equation}
M_g \stackrel{\mathrm{vec.sp.}}{\simeq}
\mbox{$\bigwedge$}(\hat{\mathfrak{h}}^f [g]_+ \oplus   \h^f_{(0)}).
\end{equation}

Let $\alpha \in \mathfrak{h}$ and $n \in \frac{1}{k}\mathbb{Z}$.  We use the notation 
\[\alpha(n)^g = \alpha \otimes t^n \in \hat{\mathfrak{h}}^f[g] .\]
Then $M_g$ is a $\hat{\mathfrak{h}}^f[g]$-module.  The action induced from the supercommutation relations (\ref{commutation-dong-zhao}) is given by 
\begin{eqnarray}
\mathbf{k}\beta (-l)^g \mathbf{1} &=&\beta (-l)^g \mathbf{1} \label{twisted-relations1-dz}\\
\alpha  (n)^g \beta (-m)^g \mathbf{1} &=& \langle \alpha_{(kn)}, \beta_{(km)} \rangle  \delta_{m,n} \mathbf{1} \\
\alpha (-n)^g \beta (-m)^g \mathbf{1} &=& -\beta(-m)^g \alpha(-n)^g \mathbf{1} \\
\alpha(0)^g \beta(0)^g \mathbf{1}  &=&  \frac{1}{2}\langle \alpha_{(0)}, \beta_{(0)} \rangle \mathbf{1}
\end{eqnarray}
for $\alpha, \beta \in \mathfrak{h}$, $m, n \in \frac{1}{k}\mathbb{Z}_+$ and $l\in \frac{1}{k}\mathbb{N}$. (Here $\mathbf{1} = 1 \in M_g$.)
As an associative algebra with identity, $M_g$ is the Clifford algebra generated by $\hat{\mathfrak{h}}^f[g]_+ \oplus \mathfrak{h}^f_{(0)}$ with the corresponding symmetric bilinear form.

For $\alpha \in \mathfrak{h}$, set
\begin{equation} \alpha(x)^g = \sum_{n \in \frac{1}{k}\mathbb{Z}} \alpha (n)^g x^{-n-\frac{1}{2}} .
\end{equation}
Then for the orthonormal basis of $\mathfrak{h}$, $\alpha^{(j)}$, for $j = 1, \dots, k$, we have 
\begin{multline}
[\alpha^{(j)}_{(km)} (x_1)^g, \alpha^{(l)}_{(kn)}(x_2)^g ] \\
= \frac{1}{k} \delta_{j,l} \delta_{m,-n}  x_1^{m+1/2}x_2^{-m-1/2}\left( \frac{1}{(x_1 - x_2)} - \frac{1}{(-x_2 + x_1)} \right)\\
\end{multline}
for $j, l = 1,\dots,k$ and $m,n \in \frac{1}{k} \mathbb{Z}$  implying that the $\alpha^{(j)}(x)^g$, for $j = 1,\dots, k$, are mutually local. 

For $v \in V_{fer}^{\otimes k}$, define $Y^g(v,x) : M_g \longrightarrow M_g [[x^{1/k}, x^{-1/k}]]$ as follows: For $\alpha \in \mathfrak{h}^f_{(r)}$, $n \in \mathbb{N} + \frac{1}{2}$, and $u \in V_{fer}^{\otimes k}$, let
\begin{multline}\label{define-tau-twisted}
Y^g (\alpha(-n) u,x) = Y^g(\alpha_{-n -1/2} u,x) = \mathrm{Res}_{x_1} \mathrm{Res}_{x_0} \left( \frac{x_1 - x_0}{x} \right)^{r/k} x_0^{-n-1/2}\\
\cdot
\left(  x^{-1}_0\delta\left(\frac{x_1-x}{x_0}\right)
\alpha (x_1)^g Y^g(u,x)
-  (-1)^{|u|} x^{-1}_0\delta\left(\frac{x-x_1}{-x_0}\right) Y^g
(u,x) \alpha (x_1)^g  \right) .
\end{multline}
Then since $V_{fer}^{\otimes k} = \langle \alpha^{(j)}(-1/2) \mathbf{1} \; | \; j = 1,\dots, k \rangle$, equation (\ref{define-tau-twisted}) defines $Y^g(v,x)$ iteratively for any $v \in V_{fer}^{\otimes k}$. 

Recalling that the Virasoro element, $\omega_{fer}$, for the free fermionic vertex operator superalgebra $V_{fer}^{\otimes d}$ is given by (\ref{fermionic-Virasoro}), a nontrivial computation shows that 
\begin{multline}\label{L-g-twisted}
L^g(m) = k \sum_{r = 0}^{k-1}  \sum_{{n \in \mathbb{Z}}\atop{n > -\frac{m}{2}}} \left( n+\frac{m}{2} - \frac{r}{k} \right) \alpha^{(1)}_{(r)}(-n+r/k) \alpha^{(1)}_{(-r)}(n+m-r/k) \\ + \frac{k^2+2}{48k} \delta_{m,0}.
\end{multline}
One should compare this with (\ref{L-sigma-twisted}).

From (\ref{L-g-twisted}) it follows that $\left[L^g(-1), Y^g ( \alpha^{(j)} (-1/2) \mathbf{1}, x)\right] = \frac{d}{dx} Y^g(\alpha^{(j)} (-1/2) \mathbf{1}, x)$, and thus, following \cite{Li-twisted}, $M_g$ is a weak $g$-twisted $V^{\otimes k}_{fer}$-module.  It is also admissible.

Similarly to the situation in the parity-twisted case, the  admissible $g$-twisted module, $M_g$, reduces as the direct sum of two irreducible parity unstable admissible $g$-twisted modules, and according to\cite{DZ},  these two irreducibles are the only irreducible admissible $g$-twisted  modules for $V_{fer}^{\otimes k} $, up to isomorphism.

Note that $M_g = \bigwedge\left[ \alpha^{(1)}_{(r)} (-n)^g1\; \big| \; r = 0, \dots, k-1, \; n \in  \mathbb{N} + \frac{r}{k} \right]$.   Thus setting $\alpha = k\alpha^{(1)}_{(0)} = \alpha^{(1)} - \alpha^{(2)} + \alpha^{(3)} + \cdots + \alpha^{(k-1)} - \alpha^{(k)}$, then $\h^f_{(0)} = \h^{0*} = \mathbb{C} \alpha$.    Let $\epsilon = \frac{1}{\sqrt{2k}} \alpha$ so that $\h^f_{(0)} = \mathbb{C} \epsilon$ and $\langle \epsilon, \epsilon \rangle = 1$.   Set 
\begin{equation}
 W = \mbox{$\bigwedge$} \left[ \epsilon (-m)^g 1 \, \Big| \, \mbox{for $m \in \mathbb{Z}_+$, }\right] 
 \end{equation}
and let $W = W^0 \oplus W^1$ be the decomposition of $W$ into even and odd subspaces.  Then  
these parity-unstable subspaces of the irreducible parity-stable module $M_g$ are given by
\begin{multline}
M_g^\pm  = \left(\left(1 \pm  \, \epsilon(0)^g 1\right) W^0 \oplus \left(1 \mp \, \epsilon(0)^g 1\right) W^1\right) \\
\otimes \mbox{$\bigwedge$} \left[\alpha^{(1)}_{(r)}(-n)^g1 \; \Big| \; r = 1, \dots, k-1, \; n \in  \mbox{$\mathbb{N} + \frac{r}{k}$}, \right] .
\end{multline}
Then we have $M_g = M_g^- \oplus M_g^+$ is an ordinary parity-stable irreducible $g$-twisted $V^{\otimes k}_{fer}$-module and $M_g^\pm$ are parity unstable invariant subspaces of $M_g$, i.e. parity unstable irreducible $g$-twisted $V^{\otimes k}_{fer}$-modules.

From (\ref{L-g-twisted}), we have that the $L^g(0)$ grading on $M_g$ is given by 
\begin{equation}
\mathrm{wt} \, \mathbf{1}  =  \frac{k^2+2}{48k},  \quad \mathrm{and} \quad  \mathrm{wt} \, \alpha^{(1)}_{(r)} (-n)^g  \mathbf{1}= n + \frac{k^2+2}{48k}, 
\end{equation}
for $n \in \mathbb{N} + \frac{r}{k}$ and $r =0, \dots, k-1$. Thus the  graded dimension of $M_g$ is 
\begin{equation}
\mathrm{dim}_q M_g \ =\  2 q^{-k/48} q^{(k^2+2)/(48 k)} \prod_{n\in \frac{1}{k}\mathbb{ Z}_+} (1 + q^n) \ = \ \sqrt{2} \mathfrak{f}_2 (q^{1/k}).
\end{equation}

\subsection{A conjecture for $(1 \; 2 \; \cdots \; k)$-twisted $V^{\otimes k}$-modules for $k$ even and $V$ any vertex operator superalgebra}

We make the following two observations, Remarks \ref{compare-grading-remark} and \ref{parity-stable-fermion-remark}, to motivate the conjecture we are about to make.

\begin{rema}\label{compare-grading-remark}
{\em 
Comparing the graded dimension of the $\sigma$ twisted $V_{fer}$-module, $M_\sigma$, to the graded dimension of the $(1\; 2 \; \cdots \; k)$-twisted $V^{\otimes k }_{fer}$ module, $M_g$, we have that 
\begin{equation}
\mathrm{dim}_q M_g = \sqrt{2} \mathfrak{f}_2(q^{1/k}) =  \mathrm{dim}_{q^{1/k}} M_\sigma. 
\end{equation}
This relationship of $q\rightarrow q^{1/k}$ between graded dimensions was in the past observed in the vertex operator algebra setting between untwisted $V$-modules and $(1 \; 2 \; \cdots \; k)$-twisted $V^{\otimes k}$-modules for $k$ even or odd, and was one of the original motivations to the proof that these two categories of modules are in fact isomorphic given in  \cite{BDM}.  That is, it had been observe that the graded dimension of a $(1 \; 2 \; \cdots \; k)$-twisted $V^{\otimes k}$-module was the same as the graded dimension of a $V$-module but with $q$ replaced by $q^{1/k}$.   In \cite{B-superpermutation-odd}, the first author showed that in the case when $V$ is a vertex operator superalgebra, the extension of the construction in \cite{BDM} to an isomorphism of categories between untwisted $V$-modules and $(1 \; 2\; \cdots \; k)$-twisted $V^{\otimes k}$-modules exists in general for a vertex operator superalgebra $V$ if and only if $k$ is an odd positive integer.  Motivated by the relationship between $M_g$ and $M_\sigma$ observed here, we make the conjecture below, Conjecture \ref{first-conjecture}, that for $k$ even, the categorical correspondence is between $(1 \; 2 \; \cdots \; k)$-twisted $V^{\otimes k}$-modules and parity-twisted $V$-modules.
}
\end{rema}

\begin{rema}\label{parity-stable-fermion-remark}
{\em In addition, we have that both the modules $M_\sigma$ and $M_g$ split into parity-unstable invariant subspaces.   This was another one of the motivating examples for Theorem \ref{parity-stability-theorem} as well as further evidence to bolster our conjecture given below, Conjecture \ref{first-conjecture}.
}
\end{rema}

These two observations given above, as well as recent constructions given by the second author, and certain observations given by the first author in \cite{B-superpermutation-odd}, leads us to the following conjecture.

\begin{conj}\label{first-conjecture}
If $V$ is a vertex operator superalgebra, and $k$ is an even positive integer, then the category of weak (parity-stable) $(1 \; 2 \; \cdots \; k)$-twisted $V^{\otimes k}$-modules for $k$ even is isomorphic to the category of weak (parity-stable) parity-twisted $V$-modules.  In addition, the subcategories of weak admissible and  ordinary $(1 \; 2 \; \cdots \; k)$-twisted $V^{\otimes k}$-modules are isomorphic to the subcategories of weak admissible and ordinary parity-twisted $V$-modules, respectively.  Furthermore, all the various subcategories of parity-unstable invariant subspaces, i.e. parity-unstable submodules, coincide.  
\end{conj}

Note that this is in contrast to the results of the first author in \cite{B-superpermutation-odd} where we prove the following:

\begin{thm}\label{my-theorem}(\cite{B-superpermutation-odd}) If $V$ is a vertex operator superalgebra, and $k$ is an odd positive integer, then the category of weak (parity-stable) $(1 \; 2 \; \cdots \; k)$-twisted $V^{\otimes k}$-modules is isomorphic to the category of weak (parity-stable) $V$-modules.   In addition, the subcategories of weak admissible and ordinary $(1 \; 2 \; \cdots \; k)$-twisted $V^{\otimes k}$-modules are isomorphic to the subcategories of weak admissible and ordinary $V$-modules, respectively.   Furthermore, all the various subcategories of parity-unstable invariant subspaces, i.e. parity-unstable submodules, coincide.  
\end{thm}

In addition, in \cite{B-superpermutation-odd}, an explicit construction of the weak, weak 
admissible, and ordinary $(1 \; 2 \; \cdots \; k)$-twisted $V^{\otimes k}$-modules for $k$ odd is given in terms of the weak, weak admissible and ordinary  $V$-modules.

\section{Lattice vertex operator superalgebras}\label{lattice-vosa-section}

We recall the notion of a lattice vertex operator superalgebra following the notation and terminology of \cite{FLM3} and using the setting and results of, e.g. \cite{Lepowsky1985}, \cite{FLM2}, \cite{DL1}, \cite{Xu}, and \cite{DL2}.

Let $L$ be a positive-definite integral lattice, with nondegenerate symmetric $\mathbb{Z}$-bilinear form $\langle \cdot, \cdot \rangle$.   We introduce a lattice $L$ with an isometry $\nu$, and two central extensions, $\hat{L}$ and $\hat{L}_{\nu}$. (There should be no confusion between this use of the symbol $L$ and
the operators $L(n)$ for the Virasoro Algebra). The lattice $L$ together with the central extension $\hat{L}$ will be used to construct a vertex operator super algebra $V_L$. The central extension $\hat{L}_{\nu}$ will be used in Section \ref{general-twisted-section} to construct a space $V_L^T$ on which $V_L$ acts via twisted vertex operators.  In Section \ref{corresponding-automorphism-section}, $\nu$ will be specified to the $-1$ isometry and a certain lift $\hat{\nu}$ on $\hat{L}_\nu$ to construct the twisted modules we are interested in.

Let $k$ be a fixed positive integer. The following initial assumptions and conditions are assumed. \\
1. Let $L$ be a positive definite integral lattice with nondegenerate symmetric $\Z$-valued bilinear form $\langle\cdot,\cdot\rangle$, i.e. $L$ is a finitely generated abelian group with positive definite symmetric $\Z$-bilinear form $\langle\cdot,\cdot\rangle:\, L\times L\rightarrow \Z$. \\
2. Let $\nu$ be an isometry of $L$ with period $k$ ($k$ need not be the order of $\nu$, and in fact will be a period that is not the order in the particular case in which we will be interested). \\
3. We fix a primitive $k$th root of unity $\eta$.  Set $\eta_0=(-1)^k\eta$, so that $\eta_0$ is a primitive $2k$th root of unity if $k$ is odd, and $\eta_0 =\eta$ remains a primitive $k$th root of unity if $k$ is even.   

Since $L$ is integral, we can give it a natural $\Z_2$-grading
\begin{equation}
L=L^0\cup L^1,\quad L^j=\{\alpha\in L\,|\,\ \langle  \alpha, \alpha \rangle \in 2 \mathbb{Z}+j\},
\end{equation}
and $L^0$ is an even sublattice of $L$.  We will use the notation $|\alpha| = j$ for $\alpha \in L^j$.

Note that 
\begin{equation}\label{condition1}
\sum_{j = 0}^{k-1} \langle \nu^j \alpha, \alpha \rangle \in \left\{ 
\begin{array}{ll}
|\alpha| + 2\mathbb{Z} & \mbox{if $k$ is odd}\\
|\alpha| + \langle \nu^{k/2} \alpha , \alpha \rangle + 2 \mathbb{Z} & \mbox{if $k$ is even} 
\end{array} \right. .
\end{equation}
In addition, 
\begin{equation}\label{condition2}
\sum_{j = 0}^{k-1} \langle j \nu^j \alpha, \alpha \rangle \in \left\{ 
\begin{array}{ll}
k\mathbb{Z} & \mbox{if $k$ is odd}\\
\frac{k}{2}  \langle \nu^{k/2} \alpha , \alpha \rangle + k \mathbb{Z} & \mbox{if $k$ is even} 
\end{array} \right. .
\end{equation}

\begin{rema}\label{doubling-k-remark}
{\em
If $k$ is even and $\langle \nu^{k/2} \alpha , \alpha \rangle \in 2\mathbb{Z} + |\alpha|$, which can always be arranged by doubling $k$ if necessary, then the expressions in \ref{condition1} and \ref{condition2} are in $|\alpha| + 2\mathbb{Z}$ and $k\mathbb{Z}$, respectively.  For the purposes of this paper, we always assume that if $k$ is even, then $\langle \nu^{k/2} \alpha , \alpha \rangle \in 2\mathbb{Z} + |\alpha|$.  That is we do indeed double $k$ if necessary.  However in the setting of permutation-twisted modules for lattice vertex operator superalgebras, this can not be done.  That is, following but extending \cite{BHL}, taking $L$ to be the orthogonal sum of $k$ copies of $K$ for $k$ even and considering $\nu = (1 \; 2 \; \cdots \; k)$ acting on $L$ in the natural way, then we have $\langle \nu^{k/2} \alpha , \alpha \rangle \in 2\mathbb{Z}$.  But doubling $k$ results in a lift that is of order $2k$, i.e. that is not the permutation automorphism on the tensor product lattice vertex operator superalgebra.  This is another  illustration of the fundamental difference between the nonsuper case or the super case for $k$ odd versus the super case for $k$ even in the permutation twisted setting.   }
\end{rema}

Let $q=k$ if $k$ is even, and  $q=2k$ if $k$ is odd.   We define the $\nu$-invariant functions
\begin{eqnarray}\label{commutator-definitio n-0}
C_0: L \times L &\longrightarrow& \C^\times \\
(\alpha, \beta) &\mapsto& (-1)^{\langle \alpha,\alpha\rangle\langle\beta,\beta\rangle+ \langle \alpha,  \beta \rangle} \nonumber,
\end{eqnarray}
and
\begin{eqnarray}\label{commutator-definition}
C: L \times L &\longrightarrow& \C^\times \\
(\alpha, \beta) &\mapsto& (-1)^{\langle \alpha,\alpha\rangle\langle\beta,\beta\rangle+\sum_{j = 0}^{k-1} \langle \nu^j \alpha,  \beta
\rangle} \eta^{\sum_{j = 0}^{k-1} \langle j \nu^j \alpha , \beta \rangle}
\nonumber \\
& &\quad\quad  =(-1)^{\langle \alpha,\alpha\rangle\langle\beta,\beta\rangle}\prod_{j=0}^{k-1} (-\eta^{j} )^{\langle \nu^j \alpha, \beta
\rangle}. \nonumber
\end{eqnarray}

Note that $C_0$ and $C$ are bilinear into the abelian group $\mathbb{C}^\times$; i.e.,
\[
C(\alpha + \beta, \gamma) = C(\alpha,\gamma)C(\beta, \gamma) \quad \mbox{and} \quad
C(\alpha, \beta + \gamma) = C(\alpha, \beta)C(\alpha, \gamma)\]
for $\alpha, \beta, \gamma \in L$, and similarly for $C_0$.  In addition, we have $C_0 (\alpha, \alpha) = 1$, and by (\ref{condition1}) and (\ref{condition2}),
we have $C(\alpha, \alpha) = 1$. Moreover, $C(\beta, \alpha) = C(\alpha, \beta)^{-1}$.

The maps $C_0$ and $C$ determine uniquely (up to equivalence) two central
extensions of $L$ by the cyclic group $\langle \eta_0 \rangle$,
\begin{equation}\label{exact-0}
1 \rightarrow \langle \eta_0 \rangle \rightarrow \hat{L} \bar{\longrightarrow} L
\rightarrow 1,
\end{equation}
\begin{equation}\label{exact}
1 \rightarrow \langle \eta_0 \rangle \rightarrow \Lnu \bar{\longrightarrow} L
\rightarrow 1,
\end{equation}
with commutator maps $c_0$ and $c_0^{\nu}$, respectively, i.e., such that
\begin{eqnarray}\label{commutator=C0}
aba^{-1} b^{-1} &=& C_0( \bar{a}, \bar{b})  \qquad \mathrm{for} \quad a,b \in
\hat{L} ,\\
aba^{-1} b^{-1} &=& C( \bar{a}, \bar{b})  \qquad \mathrm{for} \quad a,b \in
\Lnu .
\end{eqnarray}
There is a natural set-theoretic identification (which is not an
isomorphism of groups unless $k = 1$ or $k = 2$) between the groups
$\hat{L}$ and $\Lnu$ such that the respective group multiplications
$\times$ and $\times_\nu$ are related by
\begin{equation}\label{identify-central-extensions}
a \times b = \prod_{0<j<k/2} (- \eta^j)^{\langle \nu^{-j} \bar{a}, \bar{b}
\rangle} a \times_\nu b \qquad \mathrm{for} \quad a,b \in \hat{L}.
\end{equation}
Note that this is the exact same relationship as in the even lattice case treated in \cite{FLM2}, \cite{Lepowsky1985}, and \cite{BHL}. Observe further that since $C_0$ is $\nu$-invariant, if we replace the
map $\ \bar{} \ $ in (\ref{exact-0}) by $\nu \circ \ \bar{} \ $, we
obtain another central extension of $L$ by $\langle \eta_0 \rangle$
with commutator map $C_0$.  By uniqueness of the central extension of
$L$, there is an automorphism $\hat{\nu}$ of $\hat{L}$ (fixing
$\eta_0$) such that $\hat{\nu}$ is a lifting of $\nu$, i.e., such that
\begin{equation}\label{nu-on-hatL-0}
(\hat{\nu} a)^{\bar{}} = \nu \bar{a} \quad \mathrm{for} \quad a \in \hat{L} .
\end{equation}
The map $\hat{\nu}$ is also an automorphism of $\Lnu$ satisfying
\begin{equation}\label{nu-on-hatL}
(\hat{\nu} a)^{\bar{}} = \nu \bar{a} \quad \mathrm{for} \quad a \in \Lnu .
\end{equation}
Moreover, we may choose the lifting $\hat{\nu}$ of $\nu$ so that
\begin{equation}\label{nuhata=a}
\hat{\nu} a = a \quad \mathrm{if} \quad \nu \bar{a} = \bar{a}
\end{equation}
(see (\ref{propertyofnuhat}) below).

We now use the central extension $\hat{L}$ to construct a vertex operator
superalgebra $V_L$ equipped with an automorphism $\hat{\nu}$ of
period $k$, induced from the automorphism $\hat{\nu}$ of $\hat{L}$.
 This is essentially a specialized case of the ``unrelativised operators'' in Section 2 of \cite{DL1}, \cite{DL2} and of \cite{Xu}.

Embed $L$ canonically in the $\C$-vector space $\h = \C\otimes_\mathbb{Z} L$, and extend the $\Z$-bilinear form on $L$ to a
$\C$-bilinear form $\langle \cdot, \cdot \rangle$ on $\h$.  The corresponding affine Lie algebra is
\begin{equation}
\hat{\h} = \h \otimes \C[t,t^{-1}] \oplus \C {\bf k},
\end{equation}
with brackets determined by
\begin{equation}
[{\bf k},\hat{\h}]=0 \qquad \mbox{and} \qquad [\alpha \otimes t^m, \beta \otimes t^n]=\langle \alpha , \beta \rangle m\delta_{m+n,0}{\bf k} 
\end{equation}
for $\alpha, \beta \in{\h}$, and $m,n\in \Z$.
Then $\hat{\h}$ has a $\Z$-gradation, the {\it weight gradation},
given by
${\rm wt}\,(\alpha \otimes t^n)=-n$  and ${\rm wt}\, {\bf k}=0$,
for $\alpha \in {\h}$ and $n\in \Z$.

Set
\begin{equation}
\hat{\h}^+={\h}\otimes t \C[t] \ \ {\rm and} \ \ \hat{\h}^-={\h}\otimes t^{-1}
\C[t^{-1}].
\end{equation}
The subalgebra $\hat{\h}_{\Z} = \hat{\h}^+\oplus\hat{\h}^-\oplus \C {\bf k}$
of $\hat{\h}$ is a Heisenberg algebra, in the sense that its
commutator subalgebra equals its center, which is one-dimensional.
Consider the induced $\hat{\h}$-module, irreducible even under
$\hat{\h}_{\Z}$, given by
\begin{equation}
M(1)=U(\hat{\h})\otimes_{U( \h \otimes \C [t] \oplus \C {\bf k})} \C \simeq
S(\hat{\h}^-)\ \ \ (\mbox{linearly}),
\end{equation}
where ${\h}\otimes \C [t]$ acts trivially on $\C$ and ${\bf k}$ acts as 1,
$U(\cdot)$ denotes universal enveloping algebra and $S(\cdot)$ denotes
symmetric algebra.  The $\hat{\h}$-module $M(1)$ is $\Z$-graded so
that wt\,1\ =\ 0 (where we write 1 for $1\otimes 1$)
\begin{equation}
M(1)=\coprod_{n\in \mathbb{N}} M(1)_n,
\end{equation}
where $M(1)_n$ denotes the homogeneous subspace of weight $n$.

Form the induced $\hat{L}$-module and $\C$-algebra
\begin{equation}
\C \{L\}  \ = \  \C [\hat{L}] \otimes_{\mathbb{C}[ \langle \eta_0 \rangle] } \C \ 
 \simeq  \ \C [L] \qquad  (\mbox{linearly}), 
\end{equation}
where $\C [\cdot]$ denotes group algebra.  For $a \in \hat{L}$, write
$\iota(a)$ for the image of $a$ in $\C \{L\}$. Then the action of $\hat{L}$
on $\C\{L\}$ is given by
\begin{equation}
a\cdot\iota(b)=\iota(a)\iota(b)=\iota(ab)
\end{equation}
for $a,b\in\hat{L}$. We give $\C \{L\}$ the $\C$-gradation determined by
\begin{equation}
\mbox{wt}\,\iota(a)=\frac{1}{2}\langle \bar{a},\bar{a}\rangle \ \ \ \
\mbox{for}\ \
a\in \hat{L}.
\end{equation}
Also define a grading-preserving action of ${\h}$ on $\C\{L\}$ by
\begin{equation}
h\cdot\iota(a)= \langle h,\bar{a}\rangle \iota(a)
\end{equation}
for $h\in{\h}$, and define
\begin{equation}
x^h\cdot\iota(a) = x^{\langle h,\bar{a}\rangle }\iota(a)
\end{equation}
for $h\in{\h}.$  Set
\begin{equation}
V_L  \ = \ M(1)\otimes_{\C} \C \{L\}  \ \simeq  \  S(\hat{\h}^-)\otimes \C [L]  \qquad (\mbox{linearly}) 
\end{equation}
and give $V_L$ the tensor product $\C$-gradation
\begin{equation}
V_L=\coprod_{n\in \C}(V_L)_n.
\end{equation}
We have wt$\,\iota(1)=0,$ where we identify $\C \{L\}$ with $1\otimes
\C \{L\}$.  Then $\hat{L}$, $\hat{\h}_{\Z}$, $h$, $x^h$ $(h\in{\h})$
act naturally on $V_L$ by acting on either $M(1)$ or $\C \{L\}$ as
indicated above. In particular, ${\bf k}$ acts as 1.

For $\alpha \in \h$, $n\in \Z$, we write $\alpha(n)$ for the operator
on $V_L$ determined by $\alpha\otimes t^n$. For $\alpha\in \h,$ set
\begin{equation}
\alpha(x)=\sum_{n\in \Z} \alpha(n) x^{-n-1}.
\end{equation}
We use a normal ordering procedure, indicated by open colons, which
signify that the enclosed expression is to be reordered if necessary
so that all the operators $\alpha(n)$, for $\alpha\in \h$, with $n<0$, as well as the operator $a$ for $a\in \hat{L}$, are to be placed to the left of all the operators
$\alpha(n)$ and $x^{\alpha}$, for $\alpha \in \h$ and $n\ge 0$,  before the
expression is evaluated. For $a \in \hat{L}$, set
\begin{equation}
Y(a,x)= \ _\circ^\circ \ e^{\int(\bar{a}(x)- \bar{a}(0)x^{-1})} a x^{\bar{a}}
\ _\circ^\circ ,
\end{equation}
using an obvious formal integration notation.  Let $a \in \hat{L}$,
$\alpha_1, \dots, \alpha_m \in \h$, $n_1,\dots, n_m \in \Z_+$ and set
\begin{eqnarray}
v &=& \alpha_1(-n_1)\cdots \alpha_m(-n_m) \otimes \iota(a)\\
&=& \alpha_1(-n_1)\cdots \alpha_m(-n_m) \cdot \iota(a) \in V_L. \nonumber
\end{eqnarray}
Define
\begin{equation}
Y(v,x) = \ _\circ^\circ  \left(\partial_{n_1-1} \alpha_1(x)\right)\cdots
\left(\partial_{n_m-1}\alpha_m(x)\right)Y(a,x)
\ _\circ^\circ ,
\end{equation}
where again, for $n \in \mathbb{N}$, we use the notation $\partial_n = \frac{1}{n!} \left(\frac{d}{dx}\right)^n$.
This gives us a well-defined linear map
\begin{equation}
V_L \rightarrow (\mbox{End}\,V_L)[[x, x^{-1}]], \quad
v \mapsto Y(v,x)=\displaystyle{\sum_{n\in \Z}}v_nx^{-n-1}.
\end{equation}
Set ${\bf 1} = 1 = 1 \otimes 1 \in V_L$ and $\omega = \frac{1}{2} \sum_{i=1}^{{\rm dim} \, {\h}} h_i(-1)h_i(-1) {\bf 1}$,
where $\{h_j\, | \, j = 1, \dots, \mathrm{dim} \, \mathfrak{h} \}$ is an orthonormal basis of $\h$. Then $V_L = (V_L, Y,
{\bf 1}, \omega)$ is a vertex operator superalgebra of central charge $c = {\rm dim} \, {\h} = {\rm rank} \, L$.
For a proof that this is a vertex operator superalgebra, see for instance Chapter 6.1 of \cite{Xu}.

\begin{rema}\label{independenceofchoices}
{\em The construction of the vertex operator superalgebra $V_L$ depends on
the central extension (\ref{exact-0}) subject to
(\ref{commutator=C0}), and hence on the choices of $k\in \Z_+$ and the
primitive root of unity $\eta$.  But it is a standard fact that $V_L$
is independent of these choices, up to isomorphism of vertex
operator superalgebras preserving the $\hat{\h}$-module structure; see for instance Proposition
6.5.5, and also Remarks 6.5.4 and 6.5.6, of \cite{LL}.  In particular,
$V_L$ as constructed above is essentially the same as $V_L$
constructed {}from a central extension of the type (\ref{exact-0})
subject to (\ref{commutator=C0}) but with the kernel of the central
extension replaced by the group $\langle \pm 1 \rangle$.  For the
purpose of constructing twisted modules, it is valuable to have this
flexibility, and we will use this property of lattice vertex
superalgebras below in Section \ref{general-twisted-section}.}
\end{rema}

\section{Twisted modules for a lattice vertex operator superalgebra and a lift of a lattice isometry}\label{general-twisted-section}

Following \cite{Lepowsky1985}, \cite{FLM2}, \cite{FLM3}, \cite{DL2}, \cite{Xu}, we recall the construction and classification of $\hat{\nu}$-twisted $V_L$-modules for a general lattice isometry $\nu$ and a lift $\hat{\nu}$.  

Following  \cite{Lepowsky1985}, but extended to integral lattices, we note that the
automorphism $\nu$ of $L$ acts in a natural way on ${\h},$ on
$\hat{\h}$ (fixing $\textbf{k}$) and on $M(1)$, preserving the gradations, and
for $u\in \hat{\h}$ and $m\in M(1),$
\begin{equation}
\nu (u\cdot m)=\nu (u)\cdot \nu(m).
\end{equation}
The automorphism $\nu$ of $L$ lifted to the automorphism $\hat{\nu}$ of
$\hat{L}$
satisfies
\begin{equation}
\hat{\nu}(h\cdot \iota(a))=\nu(h)\cdot \hat{\nu}\iota(a),
\end{equation}
for $h \in \h$ and $a \in \hat{L}$, and for $b \in \hat{L}$ we have
\begin{equation}
\hat{\nu}(\iota(a)\iota(b))=\hat{\nu} (a\cdot \iota(b)) = \hat{\nu}(a) \cdot
\hat{\nu} \iota(b) = \hat{\nu} \iota(a) \hat{\nu} \iota(b),
\end{equation}
\begin{equation}
\hat{\nu}(x^h \cdot \iota(a)) = x^{\nu(h)} \cdot \hat{\nu}\iota(a).
\end{equation}
Thus we have a natural grading-preserving automorphism of $V_L$, which
we also call $\hat{\nu}$, which acts via $\nu\otimes\hat{\nu},$ and this action
is compatible with the other actions
\begin{eqnarray}
\hat{\nu}(a\cdot v) &=& \hat{\nu}(a)\cdot\hat{\nu}(v)\\
\hat{\nu}(u\cdot v) &=& \nu(u)\cdot \hat{\nu}(v) \\
\hat{\nu}(x^h\cdot v) &=& x^{\nu(h)}\cdot \hat{\nu}(v)
\end{eqnarray}
for $a\in \hat{L}$, $u\in \hat{\h}$, $h\in {\h}$, and $v\in V_L$, so
that $\hat{\nu}$ is an automorphism of the vertex operator superalgebra
$V_L$. 

Recalling our fixed primitive $k$th root of unity $\eta$ {}from Section
\ref{lattice-vosa-section}, for $n \in \mathbb{Z}$ set
\begin{equation}\label{hngrading}
\h_{(n)} = \{ h \in \h \; | \; \nu h = \eta^{n} h \} \subset \h,
\end{equation}
so that $\h = \coprod_{n \in \mathbb{Z}/k\mathbb{Z}} \h_{(n)}$, where
we identify $\h_{(n \; \mathrm{mod} \; k)}$ with $\h_{(n)}$, for $n \in
\mathbb{Z}$. Then in general,
\begin{equation}
\h_{(n)} = \{ h + \eta^{-n} \nu h + \eta^{-2n} \nu^2 h + \cdots +
\eta^{-(k-1)n} \nu^{k-1} h \; | \; h \in \h \}.
\end{equation}
For $n \in \mathbb{Z}/k\mathbb{Z}$, denote by
\begin{equation}\label{Pn}
P_n : \h \longrightarrow \h_{(n)},
\end{equation}
the projection onto $\h_{(n)}$, and for $h \in \h$ and $n \in
\mathbb{Z}$, set $h_{(n)} = P_{(n \; \mathrm{mod}\; k)} h$.  In
general, we have that for $h \in \h$ and $n \in \Z$,
\begin{equation}\label{h_n-formula}
h_{(n)} = \frac{1}{k} \sum_{j=0}^{k-1} \eta^{-nj} \nu^j h.
\end{equation}
Viewing $\h$ as an abelian Lie algebra, consider the $\nu$-twisted
affine Lie algebra
\begin{equation}
\hat{\h}[\nu] = \coprod_{n\in\frac{1}{k} \Z} \h_{(kn)}\otimes
t^{n}\oplus \C {\bf k}
\end{equation}
with brackets determined by
\begin{equation}
[{\bf k},\hat{\h}[\nu]]=0 \quad \mbox{and} \quad [\alpha \otimes t^m, \beta \otimes t^n]=\langle \alpha , \beta \rangle
m\delta_{m+n,0}{\bf k}
\end{equation}
for $\alpha \in{\h}_{(km)}$, $\beta \in{\h}_{(kn)}$, and
$m,n\in\frac{1}{k} \Z$.

Define the {\em weight gradation} on $\hat{\h}[\nu]$ by ${\rm wt}\,(\alpha \otimes t^{n})=-n$, $ {\rm wt}\,{\bf k}=0$,
for $n\in \frac{1}{k} \Z$, $\alpha \in {\h}_{(kn)}$.  Set
\begin{equation}
\hat{\h}[\nu]^+=\coprod_{n>0}{\h}_{(kn)}\otimes t^{n},\ \ \quad
\hat{\h}[\nu]^-=\coprod_{n<0}{\h}_{(kn)}\otimes t^{n}.
\end{equation}
Now the subalgebra
\begin{equation}\label{define-non-zero-h-subalgebra}
\hat{\h}[\nu]_{\frac{1}{k} \Z}=\hat{\h}[\nu]^+\oplus
\hat{\h}[\nu]^-\oplus \C {\bf k}
\end{equation}
of $\hat{\h}[\nu]$ is a Heisenberg algebra. Form the induced
$\hat{\h}[\nu]$-module
\begin{equation}
S[\nu]=U(\hat{\h}[\nu])\otimes_{U(\coprod_{n \ge 0}{\h}_{(kn)}\otimes
t^{n}\oplus \C {\bf k})} \C \simeq S(\hat{\h}[\nu]^-) \qquad  {\rm (linearly)},
\end{equation}
where $\coprod_{n\ge 0} \h_{(kn)} \otimes t^{n}$ acts trivially on
$\C$ and ${\bf k}$ acts as 1.  Then $S[\nu]$ is irreducible under
$\hat{\h}[\nu]_{\frac{1}{k} \Z}.$

Following \cite{DL2}, Section 6, we give the module $S[\nu]$ the
natural $\mathbb{Q}$-grading (by weights) compatible with the action of
$\hat{\h}[\nu]$ and such that
\begin{equation}\label{grade-vacuum}
{\rm wt}\,1 = \frac{1}{4k^2} \sum_{j = 1}^{k-1} j (k-j) {\rm
dim}\,({\h}_{(j)}).
\end{equation}

Following Sections 5 and 6 of \cite{Lepowsky1985} extended to this setting, we have that the automorphisms of $\Lnu$ covering the identity
automorphism of $L$ are precisely the maps $\rho^* : a \rightarrow a \rho
(\bar{a})$ for a homomorphism $\rho : L \rightarrow \langle \eta_0 \rangle$. Similarly, there is a homomorphism $\rho_0 : L \cap
\h_{(0)} \rightarrow \langle \eta_0 \rangle$ such that $\hat{\nu} a =
a \rho_0(\bar{a})$ if $\nu \bar{a} = \bar{a}$.  Now $\rho_0$ can be
extended to a homomorphism $\rho : L \rightarrow \langle \eta_0
\rangle$ since the map $1 - P_0$ induces an isomorphism {}from
$L/(L\cap \h_{(0)})$ to the free abelian group $(1 - P_0)L$.
Multiplying $\hat{\nu}$ by the inverse of $\rho_0^*$ gives us an
automorphism $\hat{\nu}$ of $\Lnu$ satisfying (\ref{nu-on-hatL}) and
\begin{equation}\label{propertyofnuhat}
\hat{\nu} a = a \quad \mathrm{if} \quad \nu \bar{a} = \bar{a},
\end{equation}
as in (\ref{nuhata=a}). 

Next, we wish to construct a space $U_T$ for $\hat{L}_{\nu}$ and $\h_{(0)}$ to act upon which will be a subspace of our twisted module.  Set
\begin{eqnarray}
N &=& \{\alpha \in L \; | \; \langle \alpha ,
\h_{(0)} \rangle= 0 \} \ = \  (1-P_0)\h \cap L, \\
M &=& (1-\nu)L \subset N, \\
R &=& \{\alpha \in N \; | \; C_N(\alpha, N) = 1 \},
\end{eqnarray}
where $C_N$ denotes the map $C$ restricted to $N\times N$.  Note that $M\subset R=Z(N)$ are all subgroups of $L$, where $Z(N)$ denotes the center of $N$. Also, it is clear that $M\subset L^0$, the even sublattice of $L$.  For $\alpha \in \h$, we have $\sum_{j=0}^{k-1} \nu^j \alpha \in
\mathfrak{h}_{(0)}$, and $N\subset\displaystyle\sum_{j=1}^{k-1}\mathfrak{h}_{(j)}$ and thus for $\alpha, \beta \in N$, the commutator map $C$,
defined by (\ref{commutator-definition}) on $N$, simplifies to
\begin{equation}\label{C-on-N1}
C_N ( \alpha, \beta) = (-1)^{\langle \alpha,\alpha\rangle\langle\beta,\beta\rangle}\eta^{ \sum_{j=0}^{k-1} \langle j \nu^j \alpha, \beta
\rangle}.
\end{equation}

Denote by $\hat{Q}$ the subgroup of $\Lnu$ obtained by pulling back any subgroup $Q$ of $L$.
Then $\{a\hat{\nu}a^{-1} \, | \, a\in\hat{L}_\nu\}\subset\hat{M} \subset (\hat{L}^0)_{\nu}$.  Note that 
by (\ref{propertyofnuhat}), we have that $\{a\hat{\nu}a^{-1} \, | \, a\in\hat{L}_\nu\} \cap \langle \eta_0 \rangle = 1$.   

For $ a\in\hat{L}_\nu$ define 
\begin{equation}\label{tau-def}
\tau(a \hat{\nu} a^{-1}) = \eta^{k\langle\bar{a},\bar{a}\rangle/2 -\sum_{j=0}^{k-1}\langle \nu^j \bar{a}, \bar{a} \rangle/2} = \eta^{k\langle\bar{a},\bar{a}\rangle/2 -k\langle\bar{a}_{(0)},\bar{a}_{(0)}\rangle/2}. 
\end{equation}
In addition, for $b \in \hat{M}$,  let
\begin{equation}\label{tau-on-eta}
\tau (\eta_0^j b) = \tau(\eta^j_0 )\tau(b) = \eta^j_0 \tau(b) \quad \mbox{ for $j = 1, \dots, q$}.
\end{equation}

Then we have the following proposition:

\begin{prop} \label{key-prop}
The map $\tau: \hat{M}\rightarrow \mathbb{C}^\times$ given by (\ref{tau-def}) and satisfying (\ref{tau-on-eta}) is a well-defined group homomorphism.  
Moreover, $\tau$ is the unique group homomorphism from $\hat{M}$ to $\mathbb{C}^\times$ satisfying  (\ref{tau-def}) and  (\ref{tau-on-eta}).  In addition, if $\langle \nu^{k/2} \alpha, \alpha \rangle \in 2\mathbb{Z} + |\alpha|$ for all $\alpha \in L$, then the image of $\tau$ lies in $\langle \eta \rangle$.
\end{prop}

\begin{proof} We first show that $\tau$ is well-defined.  Suppose $a \hat{\nu} a^{-1} = b \hat{\nu} b^{-1}$.  Then $(1 - \nu) \bar{a} = (1 - \nu) \bar{b}$, which implies $(1 - \nu^j) \bar{a} = (1 - \nu^j) \bar{b}$, for $j = 1,\dots,k-1$.  Thus 
\begin{eqnarray*}
2\langle \bar{a}, \bar{a} \rangle - \langle \nu^j \bar{a}, \bar{a} \rangle - \langle \nu^{k-j} \bar{a}, \bar{a} \rangle  & = & \langle (1-\nu^j) \bar{a}, (1-\nu^j) \bar{a} \rangle \  = \  \langle (1-\nu^j) \bar{b}, (1-\nu^j) \bar{b} \rangle \\
&= &2\langle \bar{b}, \bar{b} \rangle - \langle \nu^j \bar{b}, \bar{b} \rangle - \langle \nu^{k-j} \bar{b}, \bar{b} \rangle
\end{eqnarray*}
which implies that 
\begin{equation*}
k\langle\bar{a},\bar{a}\rangle/2 -\sum_{j=0}^{k-1}\langle \nu^j \bar{a}, \bar{a} \rangle/2 = k\langle\bar{b},\bar{b}\rangle/2 -\sum_{j=0}^{k-1}\langle \nu^j \bar{b}, \bar{b} \rangle/2 .
\end{equation*}
Therefore $\tau(a\hat{\nu} a^{-1}) = \tau(b \hat{\nu} b^{-1})$, proving that $\tau$ is well defined.

For $a, b \in \hat{L}_\nu$, we have
\begin{align}
\tau(a\hat{\nu}a^{-1})\tau(b\hat{\nu}b^{-1})&=\eta^{ k\langle\bar{a},\bar{a}\rangle/2 -\sum_{j=0}^{k-1}\langle \nu^j \bar{a}, \bar{a} \rangle/2 +   k\langle\bar{b},\bar{b}\rangle/2 -\sum_{j=0}^{k-1}\langle \nu^j \bar{b}, \bar{b} \rangle/2} \nonumber \\
&=\eta^{k\langle \bar{a}+\bar{b},\bar{a}+\bar{b}\rangle/2  -\sum_{j=0}^{k-1}\langle \nu^j( \bar{a}+\bar{b}), \bar{a}+\bar{b} \rangle/2 +\sum_{j=0}^{k-1}\langle \nu^j \bar{a}, \bar{b} \rangle } \nonumber\\
&=C(\bar{a}-\nu \bar{a},\bar{b})\tau((ba)\hat{\nu}(ba)^{-1}) \nonumber \\
&=C(\bar{a}-\nu \bar{a},\bar{b})\tau(C(\bar{a} - \nu \bar{a} , \bar{b})^{-1} (a\hat{\nu}a^{-1})(b\hat{\nu}b^{-1}) ) \nonumber \\
&=\tau((a\hat{\nu}a^{-1})(b\hat{\nu}b^{-1})). \nonumber 
\end{align}
This proves $\tau$ is a group homomorphism.  Since $\hat{M}$ is the subgroup of $\hat{L}_\nu$ which is a lift of $M$, the uniqueness follows immediately from (\ref{tau-def}) and (\ref{tau-on-eta}).  

The last statement follows from (\ref{condition1}).
\end{proof}

Next we extend $\tau$ to $\hat{R}$, and then to a maximal abelian subgroup $\hat{J}$ of $\hat{N}$.  We first observe that if $\alpha \in N$, then there exists $h \in \mathfrak{h}$, such that 
\begin{equation}
k\alpha \ = \ k h - kh_{(0)}  \ = \ kh - \sum_{j = 0}^{k-1} \nu^j h 
\ = \ \sum_{j = 1}^{k-1} (h - \nu^j h).
\end{equation}
Furthermore for $j = 1, \dots, k$, we have
$h - \nu^j h = (h- \nu h) + (\nu h - \nu^2 h) + \cdots + (\nu^{j-1} h - \nu^j h) \in (1-\nu)\mathfrak{h}$.    Therefore $k\alpha \in (1-\nu)\mathfrak{h}$.  Writing $h = c\beta$ for $c \in \mathbb{C}$ and $\beta \in L$, we have that $k \alpha \in L$ and $k \alpha = c(k\beta - \sum_{j = 0}^{k-1} \nu^j \beta)$.  It follows that $c \in \mathbb{Z}$ and thus $k\alpha \in (1-\nu) L = M$.   
That is
\begin{equation}
kN \subset (1-\nu)L = M.
\end{equation} 
Therefore $N/M$ is a finitely generated torsion group, i.e. it is finite.   Thus $R/M$ is a finite group.  Also $N/M$ finite implies that $\hat{N}/\hat{M}$ and $\hat{N}/\mathrm{ker}\,\tau$ are finite as well.   (The last statement follows from the fact that $\tau( a^k\hat{\nu}a^{-k}) = 1$ for all $a\hat{\nu}a^{-1}\in \hat{M}$.)

We wish to construct an irreducible $\hat{N}$-module, $T$, on which $\hat{M}$ acts as multiplication by the character $\tau$. 

The following is just a restatement of Proposition 6.2 of \cite{Lepowsky1985}, but extended to our setting, and follows directly from Theorem 5.5.1 of \cite{FLM3}.

\begin{prop} \label{extend-tau-prop}
There are exactly $|R/M|$ extensions of $\tau$ to a homomorphism $\chi:\hat{R}\rightarrow\mathbb{C}^{\times}$. For each such $\chi$, there is a unique (up to equivalence) irreducible $\hat{N}$-module on which $\hat{R}$ acts according to $\chi$, and every irreducible $\hat{N}$-module on which $\hat{M}$ acts according to $\tau$ is equivalent to one of these. Every such module has dimension $|N/R|^{1/2}$.  Supposing that $T$ is an irreducible module for $\hat{M}$ such that $\hat{M}$ acts as $\tau$, to construct the $\hat{N}$-module structure for $T$ corresponding to $\chi$, let $J$ be any subgroup of $N$ (necessarily containing $R$) that is maximal such that $C_N$ is trivial on $J$.  Then $\hat{J}$ is a maximal abelian subgroup of $\hat{N}$. Let $\psi:\hat{J}\rightarrow \mathbb{C}^\times$ be any homomorphism extending $\chi$ and denote by $\C_{\psi}$ the $\hat{J}$-module $\mathbb{C}$ with character $\psi$. Then $T$ is isomorphic to the induced $\hat{N}$-module 
\begin{equation}
T=\C[\hat{N}]\otimes_{\C[\hat{J}]}\C_\psi\simeq \C[N/J] \qquad \text{ (linearly)}.
\end{equation}
\end{prop}

Let $T$ be any $\hat{N}$-module on which $\hat{M}$ acts as multiplication by the character $\tau$ as given by Proposition \ref{extend-tau-prop}.  Form the induced $\Lnu$-module \begin{equation}
U_{T} = \C[\Lnu] \otimes_{\C[\hat{N}]}T. 
\end{equation}
Since $T$ can be viewed as a module for the finite group $\hat{N}/\mathrm{ker}\,  \tau$, we have that $T$ is completely reducible.  Then the structure of $T$ follows from Proposition \ref{extend-tau-prop}, and in the irreducible case,
\begin{equation}
U_T = \C[\Lnu] \otimes_{\C[\hat{N}]}T=\C[\Lnu] \otimes_{\C[\hat{J}]}\C_{\psi} \simeq \C[L/J]  \qquad \text{ (linearly)}.
\end{equation}
The action of $\Lnu$ on $U_T$ is given by
\begin{equation}
a \cdot b \otimes r  =  ab \otimes r,
\end{equation}
for $a,b \in \Lnu$, and $r\in T$, and of course 
\begin{equation}\label{M-action}
(a \hat{\nu} a^{-1}) \cdot b \otimes r = C(\bar{a} - \nu \bar{a}, \bar{b}) (b (a\hat{\nu} a^{-1})) \otimes r = C(\bar{a} - \nu \bar{a}, \bar{b}) b \otimes \tau(a \hat{\nu} a^{-1}) r.
\end{equation} 

Let $\hat{\lambda} \in \mathfrak{h}_{(0)}$ be any fixed element such that
\begin{equation}\label{extra factor}
\langle\alpha,\hat{\lambda}\rangle\in \frac{1}{k}\mathbb{Z}
\end{equation}
for all $\alpha \in L$.  

Define the following action of $\h_{(0)}$ on $U_{T}$ by
\begin{equation}
h \cdot b \otimes r = \langle h, \bar{b}+ \hat{\lambda} \rangle b \otimes r
\end{equation}
for $a,b \in \Lnu$, $r\in T$, $h \in \h_{(0)}$.   Then as operators on $U_{T}$,
\begin{equation}\label{h-operates}
ha=a(\langle h,\bar a\rangle+h)
\end{equation}
for $a \in \Lnu$ and $h \in \h_{(0)}$.  

\begin{rema}{\em
For instance, we could take $\hat{\lambda} =  0$, but in general it can be nonzero, and this gives us other $\h_{(0)}$-module structures, and will result in other $\hat{\nu}$-twisted $V_L$-module structures.  However, for the specialized case we are interested in, since $\mathfrak{h}_{(0)} = 0$ (see (\ref{hn-specialized})), this $\hat{\lambda}$ will be zero.
}
\end{rema}

Note that the projection map $P_0$ (recall (\ref{Pn})) induces an isomorphism {}from $L/N$ to $P_{0}L$, and thus we have a natural isomorphism
\begin{equation}
U_T = \C[P_{0}L]\otimes_{\mathbb{C}}T,
\end{equation}
of $\mathfrak{h}_{(0)}\cup \hat{L}_\nu$-modules.  We extend $U_T$ to a $\hat{\mathfrak{h}}[\nu]$-module by letting $\hat{\mathfrak{h}}[\nu]_{\frac{1}{k}\Z}$ (recalling (\ref{define-non-zero-h-subalgebra})) act trivially.

\begin{rema}{\em In the case that $R=N$, we have a linear isomorphism 
$U_T\simeq \C[P_0 L]$.  Also
$P_{0}L = \frac{1}{k} \left( L \cap \h_{(0)} \right)$, and
so in the case when $R=N$ we have
$U_T \simeq \C\left[\frac{1}{k} \left(L \cap \h_{(0)}\right)\right]$.
This is the case in, for instance, the important setting of permutation-twisted modules for lattice vertex operator superalgebras \cite{BDM}, \cite{BHL}.
}
\end{rema}

Now note that we can write
\begin{equation}
U_T=\coprod_{\alpha \in P_0 L} U_\alpha,
\end{equation}
where
\begin{equation}
U_\alpha =\{u \in U_T \; | \; h \cdot u = \langle h, \alpha+ \hat{\lambda} \rangle u \ \ {\rm
for \ \ } h \in \h_{(0)}\},
\end{equation}
and the actions of $\hat{L}_\nu$ and $\mathfrak{h}_{(0)}$ are compatible in the sense that 
\begin{equation}
a \cdot U_\alpha \subset U_{\alpha +\bar a_{(0)}}
\end{equation}
for $a\in \Lnu$ and $\alpha \in P_0 L$.

We define an End\,$U_T$-valued formal Laurent series $x^h$ for $h\in
\h_{(0)}$ as follows
\begin{equation}
x^h\cdot u = x^{\langle h,\alpha \rangle}u \ \ \ {\rm for}\ \ \alpha \in \mathfrak{h}_{(0)}
\ \ {\rm and} \ \ u \in U_\alpha.
\end{equation}
Then {}from (\ref{h-operates}),
\begin{equation}
x^ha=ax^{\langle h,\bar a\rangle+h}\ \ \ {\rm for}\ \ \ a\in \Lnu
\end{equation}
as operators on $U_T$.   Also, for $h \in h_{(0)},$ if $\langle h, \bar{a}_{(0)} 
\rangle\in \Z$ for all $a\in L$, define the operator $\eta^h$ on $U_T$ by
\begin{equation}\label{oct1}
\eta^h \cdot u = \eta^{\langle h,\alpha \rangle} u
\end{equation}
for $u \in U_{\alpha}$ with $\alpha \in P_{0}L$.

Then for $a\in \Lnu$, and using (\ref{M-action}), we have
\begin{equation}\label{right}
\hat\nu a =a \eta^{-\sum_{j=0}^{k-1} \nu^j \bar{a} + k \langle \bar{a} , \bar{a} \rangle /2 - \sum_{j=0}^{k-1}\langle \nu^j \bar{a}, \bar{a} \rangle/ 2} 
=a \eta^{-k\bar a_{(0)} + k\langle\bar a,\bar a \rangle/2 - k \langle\bar a_{(0)},\bar a_{(0)}\rangle/2} 
\end{equation}
as operators on $U_T$.  

Then we have
\begin{equation}
\hat{\nu}^j a   = a \eta^{-jk\bar a_{(0)} + jk\langle\bar a,\bar a \rangle/2 - jk \langle\bar a_{(0)},\bar a_{(0)}\rangle/2} 
\end{equation}
and thus
\begin{equation}\label{nuhat^k=1}
\hat{\nu}^k a = a,
\end{equation}
for all $a \in \hat{L}_\nu$ acting as operators in $\mathrm{End} \, U_T$, where we recall that we had from the beginning doubled $k$ if necessary (see Remark \ref{doubling-k-remark}).  And thus $\hat{\nu}^k = 1$ on $\hat{L}_\nu$ as well. 

It is shown in, for instance, \cite{Xu} in Chapter 6.2, that $U_T$ is an irreducible $\hat{L}_{\nu}\cup \h_{(0)}$-module when $T$ is irreducible.

Define a $\C$-gradation on $U_T$ by
\begin{equation}\label{grade-U}
{\rm wt}\, u = \frac{1}{2}\langle \alpha , \alpha \rangle \ \ \ {\rm
for}\ \ \alpha \in P_0L \ \ {\rm and} \ \ u \in U_\alpha.
\end{equation}
Then $\hat{\nu}$ preserves this gradation of $U_T$ since $\nu (\alpha) = \alpha$ for $\alpha \in P_0L \subset \mathfrak{h}_{(0)}$.

Form the space
\begin{eqnarray}
V^T_L &=& S[\nu]\otimes U_T \\ &=&
\left(U(\hat{\h}[\nu])\otimes_{U(\coprod_{n \ge 0}{\h}_{(kn)}\otimes
t^{n}\oplus \C {\bf c})} \C\right)
\otimes \left(\C[\Lnu]  \otimes_{\C[\hat{N}]} \C_\tau\right) \nonumber \\
&\simeq & S(\hat{\h}[\nu]^-) \otimes_\C ( \C[P_{0}L] \otimes_\C T), \nonumber
\end{eqnarray}
which is naturally graded (by weights), using the weight gradations of
$S[\nu]$ and $U_T$.
We let $\Lnu,$ $\hat{\h}[\nu]_{\frac{1}{k} \Z},$ ${\h}_{(0)}$ and
$x^h$, for $h\in{\h}_{(0)}$, act on $V_L^T$ by acting on either $S[\nu]$
or $U_T$, as described above.

For $\alpha \in \h$ and $n\in \frac{1}{k} \Z$, write $\alpha^T (n)$ or
$\alpha_{(kn)} (n)$ for the operator on $V_L^T$ associated with
$\alpha_{(kn)}\otimes t^n$, and set
\begin{equation}
\alpha^T (x)=\sum_{n\in\frac{1}{k} \Z} \alpha^T (n)x^{-n-1}
=\sum_{n\in\frac{1}{k} \Z} \alpha_{(kn)} (n) x^{-n-1}.
\end{equation}

Following \cite{Lepowsky1985} and \cite{FLM2}, for $\alpha \in L$, define
\begin{equation}\label{define-sigma}
\rho(\alpha) = \left\{ \begin{array}{ll}
\displaystyle{ 2^{\langle \nu^{k/2} \alpha, \alpha \rangle/2} \prod_{0< j < k/2} (1-\eta^{-j})^{\langle \nu^j \alpha, \alpha
\rangle} } & \mbox{if $k \in
2\mathbb{Z}$} \\
\\
\displaystyle{\prod_{0 < j < k/2  } (1-\eta^{-j})^{\langle \nu^j \alpha, \alpha
\rangle} } & \mbox{if $k \in 2\mathbb{Z} + 1$}
\end{array}
\right. .
\end{equation}
Then $\rho(\nu\alpha)=\rho(\alpha)$.

Using the normal-ordering procedure described above, define the {\it $\hat{\nu}$-twisted vertex
operator} $Y^{\hat{\nu}}(a,x)$ for $a\in \hat{L}$ acting on $V_L^T$ as
follows
\begin{equation}\label{L-operator}
Y^{\hat{\nu}}(a,x)= k^{-\langle \bar{a},\bar{a}\rangle /2} \rho(\bar{a}) \
_\circ^\circ
e^{\int(\bar{a}^T (x)-\bar{a}^T(0)x^{-1})} a x^{\bar{a}_{(0)}+\langle
\bar{a}_{(0)}
,\bar{a}_{(0)}\rangle /2-\langle \bar{a},\bar{a}\rangle /2} \ _\circ^\circ .
\end{equation}
Note that on the right-hand side of (\ref{L-operator}), we view $a$ as an
element of $\Lnu$ using our set-theoretic identification between $\hat{L}$
and $\Lnu$ given by (\ref{identify-central-extensions}).

For $\alpha_1,\dots,\alpha_m \in{\h},$ $n_1,\dots,n_m \in \Z_+$ and
$v=\alpha_1(-n_1)\cdots \alpha_m(-n_m) \cdot \iota(a)\in V_L$, set
\begin{equation}
W(v,x)= \ _\circ^\circ
\left(\partial_{n_1-1}\alpha_1^T(x)\right)\cdots
\left(\partial_{n_m-1}
\alpha_m^T(x)\right)Y^{\hat{\nu}}(a,x) \ _\circ^\circ ,
\end{equation}
where the right-hand side is an operator on $V^T_L$. Extend to all
$v\in V_L$ by linearity.

Define constants $c_{mnr} \in \C$ for $m, n \in \mathbb{N}$ and $r =
0,\dots, k-1$ by the formulas
\begin{eqnarray}\label{define-c's1}
\ \ \ \ \sum_{m,n\ge 0} c_{mn0} x^m y^n &=& -\frac{1}{2} \sum_{j = 1}^{k-1}
{\rm log} \left(\frac {(1+x)^{1/k} - \eta^{-j}
(1+y)^{1/k}}{1-\eta^{-j} }\right),\\
\sum_{m,n\ge 0} c_{mnr} x^m y^n &= & \frac{1}{2}{\rm log} \left( \frac
{(1+x)^{1/k} -\eta^{-r}
(1+y)^{1/k}}{1-\eta^{-r}}\right) \ \ \mbox{for}\ \ r \ne 0. \label{define-c's2}
\end{eqnarray}
Let
$\{\beta_1,\dots, \beta_{\dim \h}\}$ be an orthonormal basis of $\h$,
and set
\begin{equation}\label{define-Delta_x}
\Delta_x = \sum_{m,n\ge 0} \sum_{r=0}^{k-1} \sum^{\dim \h}_{j=1}
c_{mnr} (\nu^{-r} \beta_j)(m) \beta_j(n) x^{-m-n} .
\end{equation}
Then $e^{\Delta_x}$ is well defined on $V_L$ since $c_{00r}=0$ for all
$r$, and for $v\in V_L,$ $e^{\Delta_x}v\in V_L[x^{-1}]$.  Note that
$\Delta_x$ is independent of the choice of orthonormal basis.  Then $\hat{\nu} \Delta_x = \Delta_x \hat{\nu}$ and hence $\hat{\nu} e^{\Delta_x} = e^{\Delta_x} \hat{\nu}$ on $V_L$.

For $v\in V_L,$ the {\em $\hat{\nu}$-twisted vertex operator} $Y^{\hat{\nu}}(v,x)$ is
defined by
\begin{equation}\label{Ynuhat}
Y^{\hat{\nu}}(v,x)=W(e^{\Delta_x}v,x).
\end{equation}
Then this yields a well-defined linear map
\begin{equation}
V_L \longrightarrow (\mbox{End}\,V^T_L)[[x^{1/k},x^{-1/k}]], \quad  v
 \mapsto  Y^{\hat{\nu}}(v,x)= \sum_{n \in \frac{1}{k}\Z}v^{\hat{\nu}}_nx^{-n-1}
\end{equation}
where $v^{\hat{\nu}}_n\in \mbox{End}\,V^T_L$.  

From \cite{DL2}, \cite{Xu} we have that $(V_L^T, Y^{\hat{\nu}})$ is an irreducible $\hat{\nu}$-twisted $V_L$-module.

\section{An isomorphism between $V_{fer} \otimes V_{fer}$ and $V_{\mathbb{Z} \alpha}$}

In this section we present an isomorphism between the two free fermion vertex operator superalgebra $V_{fer} \otimes V_{fer}$ and the lattice vertex operator superalgebra $V_{\mathbb{Z} \alpha}$ with $\langle \alpha, \alpha \rangle = 1$.  The fact that these two vertex operator superalgebras are isomorphic is commonly referred to as ``boson-fermion correspondence" \cite{Frenkel}, \cite{FFR}.   More specifically this isomorphism is a  correspondence between two free fermions and a fermion constrained to the circle  $\mathbb{R}/\mathbb{Z}\alpha$.

To express this isomorphism, we polarize our two free fermions using the transformation 
\begin{equation}\label{polarize-again}
\alpha^{\pm} = \frac{1}{\sqrt{2}} (\alpha_{(1)} \mp i \alpha_{(2)})
\end{equation} 
or equivalently
$\alpha_{(1 )}= \frac{1}{\sqrt{2}} (\alpha^+ + \alpha^-)$ and $\alpha_{(2)} = \frac{i}{\sqrt{2}} (\alpha^+ - \alpha^-)$.
This polarization puts us in the setting of \cite{FFR}.  
In keeping with \cite{B-n2moduli}, \cite{B-n2axiomatic},  \cite{B-n2twisted}, \cite{B-varna}, we call $\alpha^{\pm}$ the ``homogeneous" basis for $\mathfrak{h} = \mathrm{span}_\mathbb{C}\{\alpha_1, \alpha_2\}$.  

Consider the vertex operator subalgebra of $V_{fer} \otimes V_{fer}$ generated by the vector $\alpha^+(-1/2) \alpha^-(-1/2)$.  Denote this vertex operator algebra by $\langle \alpha^+(-1/2) \alpha^-(-1/2) \rangle$.  In addition, consider the free, rank one bosonic vertex operator algebra $V_{bos} = S(\hat{\mathfrak{h}}_+) = \langle \alpha(-1) \rangle$.  Then $V_{bos}$ is isomorphic to $\langle \alpha^+(-1/2) \alpha^-(-1/2) \rangle$ as vertex operator algebras with isomorphism given by $\alpha(-1) \mapsto \alpha^+(-1/2) \alpha^-(-1/2)$.  

Then, for $n \in \mathbb{Z}$, the spaces $V_{bos} \otimes e^{n \alpha}$ are irreducible modules for $V_{bos}$, and $V_{\mathbb{Z} \alpha} = \coprod_{n \in \mathbb{Z}} V_{bos} \otimes e^{n \alpha}$.  An isomorphism $\varphi: V_L  \longrightarrow V_{fer} \otimes V_{fer}$ is given by 
\begin{equation}\label{define-phi}
\varphi: \begin{array}{lll}
1 \otimes e^{n  \alpha} & \mapsto &  \alpha^+ (-n+ 1/2) \alpha^+ (-n+3/2) \cdots \alpha^+(-1/2) \cdot 1  \\
1 \otimes e^{-n \alpha} & \mapsto & \alpha^- (-n + 1/2) \alpha^- (-n +3/2) \cdots \alpha^-(-1/2) \cdot 1 
\end{array}
\end{equation}
for $n \in \mathbb{Z}_+$.  

Here $e^{n\alpha}$ is chosen as a section of $\hat{L}$ for convenience of notation.  That is letting $e: L=\mathbb{Z}\alpha  \longrightarrow \hat{L}$, $e: n\alpha \mapsto e_{n\alpha}$
be a section of $\hat{L}$, this choice of section allows us to identify
$\mathbb{C}\{L\}$ with the group algebra $\mathbb{C}[L]$ by the linear
isomorphism
\begin{equation}
\mathbb{C}[L]  \longrightarrow  \mathbb{C}\{L\}, \quad
e^{n\alpha}  \mapsto  \iota(e_{n\alpha}).
\end{equation}
But in a slight abuse of notation, we write $e^{n\alpha}$ for $\iota(e_{n\alpha})$.

The vertex operator subalgebra $\langle \alpha^+(-1/2) \alpha^-(-1/2) \rangle$ consists of those vectors in $V_{fer} \otimes V_{fer}$ which have an equal number of positive and negative $\alpha^\pm (-m)$ terms for $m \in \mathbb{Z}_+ - \frac{1}{2}$.  Then the modules corresponding to $V_{bos} \otimes e^n$ consist of those vectors in $V_{fer} \otimes V_{fer}$ that have $n$ more positive terms than negative terms if $n>0$ and that have $n$ more negative terms than positive terms for $n<0$.

The isomorphism $V_{fer} \otimes V_{fer} \cong V_{\mathbb{Z}\alpha}$ implies that $V_{fer}^{\otimes 2d} \cong V_{\mathbb{Z} \alpha}^{\otimes d} $ for $d \in \mathbb{Z}_+$.  That is, the lattice vertex operator superalgebra corresponding to the orthogonal rank $d$ lattice $\bigoplus_{j = 1}^d \mathbb{Z} \alpha^{(j)}$ with $\langle \alpha^{(j)}, \alpha^{(k)} \rangle = \delta_{j,k}$ is isomorphic to the $2d$ free boson vertex operator superalgebra.

\section{Construction and classification of the $(1 \; 2)$-twisted $V_{fer} \otimes V_{fer}$-modules through boson-fermion correspondence and a conjecture}\label{(12)-twisted-section}

In this section, we use the isomorphism of $V_{fer} \otimes V_{fer} \cong V_{\mathbb{Z} \alpha}$ to construct the $(1 \; 2)$-twisted $V_{fer} \otimes V_{fer}$-modules by first transferring the signed permutation automorphism $(1 \; 2)$ on $V_{fer} \otimes V_{fer}$ to the corresponding automorphism $\varphi \circ (1\; 2) \circ \varphi^{-1}$ on $V_{\mathbb{Z} \alpha}$, observing that this automorphism is a lift of the $-1$ lattice isometry, and then using the construction of such twisted modules recalled in Section \ref{general-twisted-section}.  

The transposition $(1 \; 2)$ acting as a signed permutation on $V_{fer} \otimes V_{fer}$ is given by $(1 \; 2) : u \otimes v  \mapsto  (-1)^{|u||v|} v \otimes u$ for $u,v \in V_{fer}$.  In terms of the polarization (\ref{polarize-again}), this automorphism is given by 
\[(1 \; 2):  \alpha^\pm (-1/2) \mapsto \alpha^\mp (-1/2).\]

\subsection{The automorphism $\hat{\nu} = \varphi \circ (1\; 2) \circ \varphi^{-1}$ of $V_{\mathbb{Z} \alpha}$ corresponding to $(1 \; 2)$ on $V_{fer}$ }\label{corresponding-automorphism-section} 

Let $\varphi \circ (1\; 2) \circ \varphi^{-1}$ be the automorphism of $V_{\mathbb{Z} \alpha}$ corresponding to $(1 \; 2)$ on $V_{fer} \otimes V_{fer}$.  Then this automorphism is uniquely determined by  
\begin{eqnarray}\label{defining-nu-hat}
\varphi \circ (1\; 2) \circ \varphi^{-1} : \ \ \ \ \ \ \ \ V_{\mathbb{Z} \alpha} & \longrightarrow & V_{\mathbb{Z} \alpha}\\
\alpha(-1) 1 \otimes 1 & \mapsto & - \alpha(-1)1 \otimes 1 \nonumber \\
1 \otimes e^{n \alpha} & \mapsto & (-i)^n (1 \otimes e^{-n \alpha}) .\nonumber
\end{eqnarray}
In particular, letting $\nu$ be the lattice isometry
\begin{equation}
\nu : \mathbb{Z}\alpha  \longrightarrow  \mathbb{Z} \alpha, \quad n\alpha  \mapsto  -n \alpha, \nonumber
\end{equation}
then $\hat{\nu} = \varphi \circ (1\; 2) \circ \varphi^{-1}$ is a lift of the lattice isometry $\nu$ to a central extension $\Lnu$ of $L = \mathbb{Z}\alpha$ by the cycle group $\langle i \rangle$ of order 4.  

\subsection{Constructing the $\hat{\nu}$-twisted $V_{\mathbb{Z} \alpha}$-modules}\label{nu-hat-twisted-section}

We now specialize the construction of twisted modules for a lattice vertex operator superalgebra and a lift of a lattice isometry given in Section \ref{general-twisted-section} to the following setting:\\
$\bullet$ Let $L = \mathbb{Z} \alpha$ with $\langle \alpha, \alpha \rangle = 1$.\\
$\bullet$ Let $k = 4$ and let $\eta = \eta_0 = i$.  
\\
$\bullet$ Let $\nu = -1$ on $L = \mathbb{Z} \alpha$.\\
$\bullet$ Let $\hat{\nu} = \varphi \circ (1\; 2) \circ \varphi^{-1}$.

We follow Section \ref{general-twisted-section} to construct and classify the $\hat{\nu}$-twisted $V_{\mathbb{Z} \alpha}$-modules.

We have 
\begin{equation}\label{hn-specialized}
\h_{(0)} = \h_{(1)} = \h_{(3)} = 0  \quad \mathrm{and} \quad \h_{(2)} = \h,
\end{equation}
and $C(\alpha, \beta) = 1$,
for all $\alpha, \beta \in L$.  Furthermore, we have
$N=R = L$ and $M = 2L$. 

For $ a\in\hat{L}_\nu$ we have 
\begin{equation}\label{tau-def-specialized}
\tau(a \hat{\nu} a^{-1}) = i^{2\langle\bar{a},\bar{a}\rangle } = (-1)^{\langle\bar{a},\bar{a}\rangle}. 
\end{equation}
But in addition, from (\ref{defining-nu-hat}), we have, for $a  = e^{n\alpha}$,
\begin{equation}
\tau(a \hat{\nu} a^{-1}) = \tau(e^{n \alpha} \hat{\nu} e^{-n\alpha}) = \tau(e^{n \alpha} i^n e^{n \alpha}) = i^n \tau(e^{2n\alpha}).
\end{equation}
Therefore, 
\begin{equation}\label{tau-on-2L}
\tau(e^{2n\alpha}) = (-1)^{\langle n \alpha, n \alpha \rangle} i^{-n} = (-1)^{n^2} (-i)^n = i^n.
\end{equation}

Next we extend $\tau$ to $\hat{R} = \hat{L}_\nu$, thereby constructing an irreducible $\hat{L}_\nu$-module, $T$, on which $\hat{M} = 2\hat{L}_\nu$ acts as multiplication by the character $\tau$.   From Proposition \ref{extend-tau-prop}, there are exactly $|R/M| = |L/2L| = 2$ extensions of $\tau$ to a homomorphism $\chi:\hat{L}_\nu\rightarrow\mathbb{C}^{\times}$, and every irreducible $\hat{L}_\nu$-module on which $\hat{M}$ acts as $\tau$ is equivalent to one of these.  Furthermore, since $N=R = L$, we have $\chi = \psi$ in Proposition \ref{extend-tau-prop}, and thus the modules $T = \mathbb{C}_\psi = \mathbb{C}_\chi$ will be precisely these two modules. 

It is clear from (\ref{tau-on-2L}), that the two choices for $\chi$ are 
\begin{equation}\label{define-chi}
\chi_\pm: \hat{L}_\nu  \longrightarrow  \mathbb{C}^\times, \quad e^{n\alpha}  \mapsto  \pm (e^{\pi i /4})^n 
\end{equation}
for the primitive eighth root of unity $e^{\pi i/4}$ and $n \in \mathbb{Z}$.

Denote these two inequivalent irreducible $\hat{L}_\nu$-modules on which $\hat{L}_\nu$ act as $\chi_+$ and $\chi_-$, respectively, by $\mathbb{C}_+$ and $\mathbb{C}_-$, respectively.  Then we have two choices for $U_T$ up to isomorphism, namely $U_T = \mathbb{C}_\pm$.

Note that (\ref{right}) does reduce to $\hat\nu a =a i^{2 \langle \bar{a} , \bar{a} \rangle} = a (-1)^{\langle \bar{a}, \bar{a} \rangle} = (-1)^{|a|} a$, as operators on either $\mathbb{C}_\pm$, and $\hat{\nu}^2 = 1$ on $\hat{L}_\nu$.

Form the two $\hat{\nu}$-twisted $V_{\mathbb{Z} \alpha}$ modules 
\begin{equation}
M_\pm = S[\nu]\otimes\mathbb{C}_\pm \simeq S[\nu].
\end{equation}

Note that in this setting, we have  
\begin{equation}
\alpha^T (x)=\sum_{n\in \Z + \frac{1}{2}} \alpha^T (n)x^{-n-1}.
\end{equation}

Then the $\hat{\nu}$-twisted vertex operators are given by (\ref{L-operator}) and (\ref{Ynuhat}).  We denote these two different $\hat{\nu}$-twisted vertex operators by $Y_\pm^{\hat{\nu}}$ on $M\pm$, respectively.

And note that of course we have 
\begin{equation}
V_{\mathbb{Z} \alpha} \longrightarrow (\mbox{End}\, M_\pm)[[x^{1/2},x^{-1/2}]], \quad  \ v
\mapsto Y_\pm^{\hat{\nu}}(v,x)= \sum_{n \in \frac{1}{2}\Z}v^{\hat{\nu}, \pm}_n x^{-n-1}
\end{equation}
where $v^{\hat{\nu}, \pm}_n\in \mbox{End}\, M_\pm$.  

From (\ref{grade-vacuum}), we have
\begin{eqnarray}\label{grade-vacuum-specified}
{\rm wt}\,1 &=& \frac{1}{64} \sum_{j = 1}^{3} j (4-j) {\rm
dim}\,({\h}_{(j)}) = \frac{1}{16}.
\end{eqnarray}
Following (\ref{grade-U}), we have that the $\C$-gradation on $\mathbb{C}_\pm$ is zero, and thus the weight grading of $\mathbb{C}_\pm$ is $\frac{1}{16}$.  In fact, we have that 
\begin{equation}
L^{\hat{\nu}}_\pm (0) = \sum_{n \in \mathbb{N}} \alpha^T ( - n - 1/2) \alpha^T(n + 1/2) + \frac{1}{16},
\end{equation}
and thus 
\begin{equation}
\mathrm{dim}_q M_\pm = q^{-1/24} q^{1/16} \prod_{n \in \mathbb{Z}_+} (1 + q^{n/2}) = 
\frac{\eta(q)}{\eta(q^{1/2})} = \sqrt{2}^{-1}  \mathfrak{f}_2(q^{1/2}).
\end{equation}
In particular, $M_+$ and $M_-$ have the same graded dimension.  Note also that this graded dimension is the same as the graded dimension for the $g$-twisted free boson vertex operator algebra module where $g$ is uniquely determined by $-1$ on the generator; see for instance \cite{B-n2twisted} Section 5.2.  It is also the graded dimension with $q$ replaced by $q^{1/2}$ for either of the two unique up to equivalence irreducible parity-twisted $V_{fer}$-modules as constructed in Section \ref{parity-twisted-section}; see Remark \ref{compare-grading-remark}.

From \cite{DL2}, \cite{Xu} we have that $(M_\pm, Y_\pm^{\hat{\nu}})$ are each irreducible $\hat{\nu}$-twisted $V_{\mathbb{Z} \alpha}$-modules, and they are the only irreducible $\hat{\nu}$-twisted $V_{\mathbb{Z} \alpha}$-modules.  
In addition, although these two modules, $M_+$ and $M_-$ have the same graded dimension, 
they are not isomorphic to each other as $\hat{\nu}$-twisted $V_{\mathbb{Z} \alpha}$-
modules.  That is, if $f: (M_+ , Y_+^{\hat{\nu}}) \longrightarrow (M_- , Y_-^{\hat{\nu}})$ is a 
twisted module isomorphism, then $f \circ Y_+(v, x) \circ f^{-1} = Y^{\hat{\nu}}_-(v, x) = 
(-1)^{|v|} Y^{\hat{\nu}}_+(v,x )$ for all $v \in V$.  This would imply that there exists a 
$\mathbb{Z}_2$-grading on $M_\pm$, given by $M_\pm = M^{(0)}_\pm \oplus M^{(1)}_\pm$, 
such that $|Y_+(v, x) w| = (|v| + |w|) \mathrm{mod} \, 2$ for all $w \in M_+$, where $|w| =j$ 
for $w \in M_+^{(j)}$, for $j = 0,1$.
But if $w \in M_+$ is a nonzero vector which is homogeneous with respect to the 
$\mathbb{Z}_2$-grading, then writing $w = u \otimes t$ with $u \in S[\nu]$ and $t\in 
\mathbb{C}_+$, we have
\begin{eqnarray}
Y^{\hat{\nu}}_+(e^\alpha,x)w &=& \frac{1}{2} \rho(\alpha) \ \left(
_\circ^\circ
e^{\int(\alpha^T (x)-\alpha^T(0)x^{-1})}  \ _\circ^\circ  x^{-1/2} u \right) \otimes \chi_+(e^\alpha) \cdot t \\
&=& r \left(
_\circ^\circ
e^{\int(\alpha^T (x)-\alpha^T(0)x^{-1})}  \ _\circ^\circ  x^{-1/2} u \right) \otimes t  \nonumber
\end{eqnarray}
for a constant $r \in \mathbb{C}^\times$.  
Taking $r^{-1}\mathrm{Res}_{x} x^{-1/2}$ of both sides we obtain $u \otimes t = w$.  
Since $|e^{\alpha}| = 1$, this implies that $|w| = (1 + |w|) \mathrm{mod} \, 2$, a contradiction.  Thus there exists no such $\mathbb{Z}_2$-grading that would give an isomorphism between $M_+$ and $M_-$.

It follows that

\begin{lem}
The modules $M_\pm$ are isomorphic as ordinary parity-unstable $( 1 \; 2)$-twisted $V_{fer}^{\otimes 2}$-modules to the modules $M_{(1 \; 2)}^\pm$ presented in Section \ref{permutation-dong-zhao-section} following \cite{DZ}.  
\end{lem}

\begin{rema}
{\em  We observe that using the construction of \cite{DZ}, one first constructs a parity-stable irreducible $( 1 \; 2)$-twisted $V_{fer}^{\otimes 2}$-module and then identifies two invariant subspaces which are parity-unstable irreducible $( 1 \; 2)$-twisted $V_{fer}^{\otimes 2}$-modules.  However using boson-fermion correspondence and the theory of lattice vertex operator superalgebras, one first directly constructs a pair of parity-unstable irreducible $( 1 \; 2)$-twisted $V_{fer}^{\otimes 2}$-modules.   In addition, from the lattice construction it is immediately obvious that $(M_+, Y^{\hat{\nu}}_+)$ is isomorphic to $(M_-, Y_-^{\hat{\nu}} \circ \sigma_V)$ as parity-unstable $\hat{\nu}$-twisted modules, wheres the isomorphism is less straightforward in the \cite{DZ} construction.
}
\end{rema}

\subsection{A conjecture}   For $k>2$ even, if one tries to directly lift $(1 \;2 \; \cdots \; k )$ to the lattice vertex operator superalgebra $V_{\Z \alpha}^{\otimes k}  = V_{\Z \alpha \oplus \cdots \oplus \Z \alpha}$ in the obvious way extending what we did in the case $k = 2$ in the last section, one does not get a lift of a lattice isometry.  

For example if one tries to directly lift, say $(1 \;2 \; 3 \; 4)$ to the lattice vertex operator superalgebra $V_{\Z \alpha}^{\otimes 2}  = V_{\Z \alpha \oplus \Z \alpha}$ in the obvious way, we have that for instance
\begin{multline}\label{4-cycle}
(\varphi \otimes \varphi) \circ (1 \; 2 \; 3 \; 4) \circ (\varphi \otimes \varphi)^{-1} :  (\alpha(-1)1 \otimes 1) \otimes (1 \otimes 1) \\
\mapsto \frac{1}{2} \left( \left(1 \otimes (e^{\alpha} + e^{-\alpha})\right) \otimes \left(1 \otimes (e^{\alpha} - e^{-\alpha})\right)\right).
\end{multline}

However, we note that all the  $(1 \; 2 \; \cdots \; k)$-twisted $V^{\otimes k}_{fer}$-modules, $M_{(1 \; 2 \; \cdots \; k)}^\pm$, of Section \ref{permutation-dong-zhao-section} for $k$ even, have a structure that would imply that they could be realized as $g$-twisted $V_{\Z \alpha}^{\otimes k/2}$-modules for $g$ a lift of some lattice isometry.   In particular, they look like $S[\nu] \otimes \mathbb{C}_\pm$ for $S$ a symmetric algebra and $\mathbb{C}_\pm$ a one dimensional space on which the odd generating operator acts as $\pm c$ for a constant $c$.  

This, as well as recent constructions of other permutation-twisted modules for free fermions given by the second author, lead us to the following conjecture:

\begin{conj}  The $(1 \; 2\; \cdots \; k)$ permutation automorphism of $V_{fer}^{\otimes k}$ for $k$ even is conjugate to a lift of a lattice isometry on $V_{\mathbb{Z} \alpha}^{\otimes k/2}$ via boson-fermion correspondence. 
\end{conj}

Note that from (\ref{4-cycle}), this conjecture is nontrivial.  In addition, we stress that although the permutation automorphisms for free fermions can be realized as lifts of isometries on $\mathfrak{h}$ as in \cite{DZ} and Section \ref{permutation-dong-zhao-section}, this conjecture goes further to state that they can be be realized as lifts of isometries on the lattice underlying the purely bosonic part of the vertex operator superalgebra.  This is a much stronger statement, and allows for the full theory of twisted modules for lattice vertex operator superalgebras to come to bear.  

\section{Construction and classification of the $\sigma \circ (1 \; 2)$-twisted $V_{fer} \otimes V_{fer}$-modules}\label{sigma-2-section}

Let $\hat{\nu} = \varphi \circ (1\; 2) \circ \varphi^{-1}$ be the automorphism of $V_{\mathbb{Z} \alpha}$ corresponding to $(1 \; 2)$ on $V_{fer} \otimes V_{fer}$ given explicitly by (\ref{defining-nu-hat}), where $\varphi$ is the isomorphism between $V_{fer} \otimes V_{fer}$ and $V_{\mathbb{Z} \alpha}$ given by (\ref{define-phi}).  Then the automorphism $\sigma\circ (1\;2)$ of $V_{fer} \otimes V_{fer}$ corresponds to the automorphism $\sigma \circ \hat{\nu} = \varphi \circ \sigma \circ (1\;2) \circ \varphi^{-1}$ given by   
\begin{eqnarray}\label{defining-sigma-nu-hat}
\varphi \circ \sigma \circ (1\; 2) \circ \varphi^{-1} : \ \ \ \ \ \ \ \ V_{\mathbb{Z} \alpha} & \longrightarrow & V_{\mathbb{Z} \alpha}\\
\alpha(-1) \otimes 1 & \mapsto & - \alpha(-1) \nonumber \\
1 \otimes e^{n \alpha} & \mapsto & i^n (1 \otimes e^{-n \alpha} ).\nonumber
\end{eqnarray}
Then $\sigma \circ \hat{\nu} = \varphi \circ \sigma \circ (1\; 2) \circ \varphi^{-1}$ is also a lift of the lattice isometry $\nu = -1$ to a central extension of $L = \mathbb{Z}\alpha$ by the cycle group $\langle i \rangle$ of order 4.   Repeating the construction of Section \ref{nu-hat-twisted-section}, we have that there are exactly two inequivalent irreducible $\sigma \circ \hat{\nu}$-twisted $V_{\mathbb{Z} \alpha}$ given by 
\begin{equation}
M^\sigma_\pm = S[\nu]\otimes\mathbb{C}^\sigma_\pm \simeq S[\nu].
\end{equation}
where $\mathbb{C}_\pm^\sigma$ are the irreducible $\hat{L}_{\nu}$-modules constructed as follows:  Define the characters
\begin{equation}
\chi^\sigma_\pm: \hat{L}_\nu  \longrightarrow  \mathbb{C}^\times, \quad e^{n\alpha}  \mapsto  \pm (e^{3\pi i /4})^n
\end{equation}
for the primitive eighth root of unity $e^{3\pi i/4}$, cf. (\ref{define-chi}).
Denote the two inequivalent irreducible $\hat{L}_\nu$-modules on which $\hat{L}_\nu$ act as $\chi_+^\sigma$ and $\chi_-^\sigma$, respectively, by $\mathbb{C}^\sigma_+$ and $\mathbb{C}^\sigma_-$, respectively. 

\begin{prop}
The modules $M_\pm$ are isomorphic as ordinary parity-unstable $\sigma ( 1 \; 2)$-twisted $V_{fer}^{\otimes 2}$-modules to the modules $M_{\sigma (1 \; 2)}^\pm$  following \cite{DZ}, and are the only irreducible parity-unstable $\sigma ( 1 \; 2)$-twisted $V_{fer}^{\otimes 2}$-modules up to equivalence.  
\end{prop}

We further conjecture that in general, for $k$ even, (not just two) that $\sigma \circ (1 \; 2 \; \cdots \; k)$ can be realized as a conjugate of a lift of a lattice isometry under boson-fermion correspondence.

\end{document}